\documentclass[a4paper,12pt]{amsart}

\usepackage{latexsym}
\usepackage{amssymb}
\usepackage{graphicx}
\usepackage{wrapfig}
\usepackage{amsthm}
\usepackage{float}
\usepackage{epsfig}
\usepackage{tikz}
\usetikzlibrary{arrows}
\usetikzlibrary{arrows.meta}
\usetikzlibrary{positioning,automata}
\usepackage{multirow}
\usepackage[toc,page]{appendix}
\usepackage[footnotesize]{caption}
\RequirePackage{color}

\usepackage{amscd,enumerate}
\usepackage{amsfonts}

\newtheorem{lemma}{Lemma}[section]
\newtheorem{corollary}[lemma]{Corollary}
\newtheorem{theorem}[lemma]{Theorem}
\newtheorem{proposition}[lemma]{Proposition}
\newtheorem{remark}[lemma]{Remark}
\newtheorem{definition}[lemma]{Definition}
\newtheorem{definitions}[lemma]{Definitions}
\newtheorem{example}[lemma]{Example}
\newtheorem{examples}[lemma]{Examples}

\newcommand{\Z}{{\mathbb{Z}}}

\newcommand{\uloopr}[1]{\ar@'{@+{[0,0]+(-4,5)}@+{[0,0]+(0,10)}@+{[0,0] +(4,5)}}^{#1}}

\definecolor{turquoise2}{rgb}{0,0.898039,0.933333}
\definecolor{magenta}{rgb}{1,0,1}
\definecolor{olivedrab}{rgb}{0.419608,0.556863,0.137255}
\definecolor{purple2}{rgb}{0.568627,0.172549,0.933333}
\definecolor{amethyst}{rgb}{0.6, 0.4, 0.8}
\definecolor{ao(english)}{rgb}{0.0, 0.5, 0.0}
\definecolor{atomictangerine}{rgb}{1.0, 0.6, 0.4}
\definecolor{amber(sae/ece)}{rgb}{1.0, 0.49, 0.0}
\definecolor{alizarin}{rgb}{0.82, 0.1, 0.26}
\definecolor{auburn}{rgb}{0.43, 0.21, 0.1}
\definecolor{aqua}{rgb}{0.0, 1.0, 1.0}

\begin{document}


\author[Vural Cam]{Vural Cam}
\address{Vural Cam: Department of Mathematics. Sel\c{c}uk University, Sel\c{c}uklu 42003 Konya, Turkey}
\email{cvural@selcuk.edu.tr}

\author[Crist\'obal Gil Canto]{Crist\'obal Gil Canto}
\address{Crist\'obal Gil Canto:  Departamento de \'Algebra Geometr\'{\i}a y Topolog\'{\i}a, Fa\-cultad de Ciencias, Universidad de M\'alaga, Campus de Teatinos s/n. 29071 M\'alaga. Spain.}
\email{cristogilcanto@gmail.com}

\author[M\"uge Kanuni]{M\"uge Kanuni}
\address{M\"uge Kanuni: Department of Mathematics. D\"uzce University, Konuralp 81620 D\"uzce, Turkey}
\email{mugekanuni@duzce.edu.tr}

\author[Mercedes Siles ]{Mercedes Siles Molina}
\address{Mercedes Siles Molina: Departamento de \'Algebra Geometr\'{\i}a y Topolog\'{\i}a, Fa\-cultad de Ciencias, Universidad de M\'alaga, Campus de Teatinos s/n. 29071 M\'alaga.   Spain.}
\email{msilesm@uma.es}

\thanks{
The second and the last author are supported by the Junta de Andaluc\'{\i}a and Fondos FEDER, jointly, through project  FQM-336.
They are also supported by the Spanish Ministerio de Econom\'ia y Competitividad and Fondos FEDER, jointly, through project  MTM2016-76327-C3-1-P.
\newline
This research started while the first and third authors were visiting the Universidad de M\'alaga. They both thank their coauthors for their hospitality. These authors also thank Nesin Mathematics Village, Izmir, for maintaining an excellent research environment while conducting this research in the summer of 2018.
\newline
The first author is supported by the Scientific and Technological Research Council of Turkey (T\"{U}B\.{I}TAK-B\.{I}DEB) 2019-International Post-Doctoral Research Fellowship during his visit to Universidad de M\'alaga.}

\medskip

\title[Largest ideals in Leavitt Path Algebras]{Largest ideals in Leavitt Path Algebras}

\begin{abstract} 
We identify largest ideals in Leavitt path algebras:  the  largest locally left/right artinian (which is the largest semisimple one), the largest locally left/right noetherian without minimal idempotents, the largest exchange, and the largest purely infinite. This last ideal is described as a direct sum of purely infinite simple pieces plus purely infinite non-simple and non-decomposable pieces. The invariance under ring isomorphisms of these ideals is also studied.
\end{abstract}

\maketitle

%%%%%%%%%%%%%%%%%%%%%%%%%%%%%%
%%%%%%%%%%%%%%%%%%%%%%%%%%%%%%
\section{Introduction and preliminary results}
%%%%%%%%%%%%%%%%%%%%%%%%%%%%%%
%%%%%%%%%%%%%%%%%%%%%%%%%%%%%%
%%%%%%%%%%%%%%%%%%%%%%%%%%%%%%
%%%%%%%%%%%%%%%%%%%%%%%%%%%%%%

Since they were introduced in \cite{AA1} and \cite{AMP}, Leavitt path algebras have attracted significant interest and attention. When examining the structure of a Leavitt path algebra $L_K(E)$ for a field $K$ and an arbitrary graph $E$, one can realize that four important pieces appear: these are the set of line points $P_l$, the set of vertices in cycles without exits $P_c$, the set of vertices in extreme cycles $P_{ec}$ and the set $P_{b^\infty}$ of vertices whose tree has infinitely many bifurcations or at least one infinite emitter. 

 To begin with, the ideal generated by $P_l$ was firstly studied in \cite{AMMS1, AMMS2}: it is precisely the socle of the Leavitt path algebra and it is isomorphic to a direct sum of matrix rings over $K$. Secondly, the ideal generated by $P_c$, studied in  \cite{ABS,AAPS,CMMSS}, is isomorphic to a direct sum of matrix rings over $K[x,x^{-1}]$. On the other hand, the ideal generated by $P_{ec}$, originally presented in \cite{CMMSS}, is  a direct sum of purely infinite simple rings. To highlight the importance of $P_l$, $P_c$ and $P_{ec}$, we remind that these three sets are the key ingredients in order to determine the center of a Leavitt path algebra \cite{CMMSS}.
 
 In this work we show that $I(P_l)$ (respectively $I(P_c)$), contains the information about the locally left/right artinian (respectively left/right noetherian) side of the Leavitt path algebra; more concretely, we will see that it is the largest locally left/right artinian (respectively left/right noetherian without minimal idempotents) ideal inside $L_K(E)$. As for the ideal generated by $P_{ec}$ we prove that it is purely infinite. The notion of purely infiniteness for rings was introduced in \cite{AGPS}, where the (not necessarily simple) purely infinite Leavitt path algebras were characterized too. We will see that, although the ideal generated by $P_{ec}$ is purely infinite, it is not the largest with this property. Then we will determine the largest purely infinite ideal (which will be not necessarily simple) inside $L_K(E)$. The following goal in this paper will be to find the largest exchange ideal of a Leavitt path algebra. We know that it exists by \cite[Theorem 3.5]{AMM} and here we will determine exactly which set of vertices generates it.

This paper is organized as follows. In Section \ref{Sec:ideals} we show that, for Leavitt path algebras of arbitrary graphs, the ideal generated by $P_l \cup P_c \cup P_{ec} \cup P_{b^\infty}$ is dense and that $I(P_{ec} \cup P_{b^\infty})$ is invariant under any ring isomorphism. The invariance of $I(P_l)$ and $I(P_c)$ is still known (the first ideal because it is the socle of the Leavitt path algebra, and the second one by \cite[Theorem 6.11]{ABS}). In Section \ref{Sec:largestlocart&noeth}
we prove that the ideal generated by $P_l$ is the largest locally artinian ideal of the Leavitt path algebra and that the ideal generated by $P_c$ is the largest locally noetherian one without minimal idempotents. In the next section we complete the picture about  largest ideals with a certain property: concretely we find the largest purely infinite ideal. To this aim we prove in Proposition \ref{Proposition:piisgraded} that every purely infinite ideal is graded and that, despite $I(P_{ec})$ being purely infinite, it is not the largest one inside $L_K(E)$. We then construct a new hereditary and saturated set of vertices, denoted by $P_{ppi}$, that contains $P_{ec}$ (Lemma \ref{lem:tpiher}) and which generates the largest purely infinite ideal of the Leavitt path algebra (Proposition \ref{Prop:PPIispurelyinfinite} and Theorem \ref{Theorem:piisthemaximal}). We also prove that this ideal is invariant. We devote Section \ref{Sec:structurelpi} to describe the internal structure of the ideal generated by $P_{ppi}$; in fact, we describe $I(P_{ppi})$ in Theorem \ref{theor:biggestPIideal} as a direct sum of ideals which are isomorphic to purely infinite simpe Leavitt path algebras plus ideals which are isomorphic to purely infinite not simple not decomposable Leavitt path algebras. Finally, in Section \ref{Sec:largestexchange} we identify graphically the set of vertices which generates the largest exchange ideal in a Leavitt path algebra, namely $P_{ex}$ (see Theorem \ref{thm:largestex}) and we prove that this ideal is invariant under any ring isomorphism too.

\medskip

We now present some background material. Throughout the paper $E = (E^0, E^1, s, r)$ will denote a directed graph with set of vertices $E^0$,  set of edges $E^1$, source map $s$, and range map $r$. In particular, the source vertex of an edge $e$ is denoted by $s(e)$, and the range vertex by $r(e)$. We call $E$ {\it finite} if both $E^0$ and $E^1$ are finite sets and \emph{row-finite} if $s^{-1}(v)= \{e\in E^1 \mid s(e) = v\}$ is a finite set for all $v\in E^0$. A vertex $v$ is called an {\it infinite emitter} is
$s^{-1}(v)$ is not a finite set. A {\it sink} is a vertex $v$ for which  $s^{-1}(v) $ is empty. Vertices which are neither sinks nor infinite emitters are called \textit{regular vertices}. For each $e\in E^{1}$, we call $e^{\ast}$ a \textit{ghost edge}. We let
$r(e^{\ast})$ denote $s(e)$, and we let $s(e^{\ast})$ denote $r(e)$. A \textit{path} $\mu$ of length $|\mu|=n>0$ is a finite sequence of edges $\mu
=e_{1}e_{2}\ldots e_{n}$ with $r(e_{i})=s(e_{i+1})$ for all
$i=1,\ldots,n-1$. In this case $\mu^{\ast}=e_{n}^{\ast}\ldots e_{2}^{\ast}e_{1}^{\ast}$ is the corresponding  \textit{ghost path}. A vertex is considered a path of length $0$. The set of all sources and ranges of the edges appearing in the expression of the path $\mu$
is denoted by $\mu^{0}$. When $\mu$ is a vertex, $v^0$ will denote $v$. The set of all paths of a graph $E$ is denoted by ${\rm Path}(E)$.

If there is a path from a vertex $u$ to a vertex $v$, we write $u\geq v$.
A subset $H$ of $E^{0}$ is called \textit{hereditary} if, whenever $v\in H$ and
$w\in E^{0}$ satisfy $v\geq w$, then $w\in H$. A  set $X$ is
\textit{saturated} if for any  vertex $v$ which is neither a sink nor an infinite emitter, $r(s^{-1}(v))\subseteq X$
implies $v\in X$. Given a nonempty subset $X$ of vertices, we define  its \emph{saturation}, $S(X)$, as follows
$$S(X):= \{v\in {\rm Reg}(E) \ \vert \ \{r(e) \ \vert \ s(e)=v\}\subseteq X\} \cup X.$$
The \emph{tree} of $X$, denoted by $T(X)$, is the set
$$T(X):=\{u\in E^0 \ \vert \  x\geq u \ \text{for some} \ x\in X\}.$$ This is a hereditary subset of $E^0$.  The notation $\overline{X}$ ($\overline{X}^E$ if we want to emphasize the graph $E$) will be used for the hereditary and saturated closure of $X$, which is built, for example, in \cite[Lemma 2.0.7]{AAS}.
Concretely, if $X$ is nonempty, then we define  
$X_0:=T(X)$, and for $n\geq 0$ we define inductively $X_{n+1}:=S(X_n)$. Then $\overline X = \cup_{n\geq 0}X_n$.

%The set of all hereditary saturated subsets of $E^{0}$ is denoted by $\mathcal{H}_E$, which is also a partially ordered set by  inclusion.

A path $\mu$ $=e_{1}\dots e_{n}$, with $n>0$, is \textit{closed} if $r(e_{n})=s(e_{1})$,
in which case $\mu$ is said to be \textit{based at the vertex
}$s(e_{1})$ and $s(e_{1})$ is named as the \textit{base of the path}. A closed path $\mu$ is called \textit{simple} provided that
it does not pass through its base more than once, i.e., $s(e_{i})\neq
s(e_{1})$ for all $i=2,\ldots,n$. The closed path $\mu$ is called a
\textit{cycle} if it does not pass through any of its vertices twice, that is,
if $s(e_{i})\neq s(e_{j})$ for every $i\neq j$.

An \textit{exit} for a path $\mu=e_{1}\dots e_{n}$, with $n>0$, is an edge $e$ such that
$s(e)=s(e_{i})$ for some $i$ and $e\neq e_{i}$. We say the graph $E$ satisfies
\textit{Condition} (L) if every cycle in $E$ has an exit. We say the graph $E$ satisfies
\textit{Condition} (K) if for each $v\in E^0$ which lies on a closed simple path, there exist at least two distinct closed simple paths based at $v$. We denote by $P_c^E$ the set of vertices of a graph $E$ lying in
cycles without exits.

A cycle $c$ in a graph $E$ is called an \textit{extreme cycle} if $c$ has exits and for every path $\lambda$ starting at a vertex in $c^0$ there exists $\mu \in {\rm Path}(E)$ such that $0 \ne \lambda \mu$ and $r(\lambda \mu) \in c^0$.

A vertex $v \in E^0$ is called a \textit{bifurcation} vertex (or it is said that \textit{there is a bifurcation at} $v$) if $\vert s_E^{-1}(v) \vert \geq 2$. A {\it line point} is a vertex $v$ whose tree $T(v)$ does not contain any bifurcations or cycles.
We will denote by $P^E_l$ the set of all line points, by $P^E_{ ec}$ the set of vertices which belong to extreme cycles, while $P^E_{ lec}:= P^E_l \sqcup P^E_{c} \sqcup P^E_{ ec}$. Moreover, $P_{b^{\infty}}^E$ denotes the set of all vertices $v \in E^0$ whose tree $T(v)$ contains infinitely many bifurcation vertices or at least one infinite emitter.
We will eliminate the superscript $E$ in these sets if there is no ambiguity about the graph we are considering.

Let $K$ be a field, and let $E$ be a directed graph. The {\it Leavitt path $K$-algebra} $L_K(E)$ {\em of $E$ with coefficients in $K$} is  the free $K$-algebra generated by the set $\{v\mid v\in E^0\}$, together with  $\{e,e^*\mid e\in E^1\}$, which satisfy the following relations:

(V)  $vw = \delta_{v,w}v$ for all $v,w\in E^0$, \

(E1) $s(e)e=er(e)=e$ for all $e\in E^1$,

(E2)  $r(e)e^*=e^*s(e)=e^*$ for all $e\in E^1$, and

(CK1)  $e^*e'=\delta _{e,e'}r(e)$ for all $e,e'\in E^1$.

(CK2)  $v=\sum _{\{ e\in E^1\mid s(e)=v \}}ee^*$ for every regular vertex $v\in E^0$.

We refer the reader to the book \cite{AAS} for other definitions and results on Leavitt path algebras.
\medskip

%%%%%%%%%%%%%%%%%%%%%%%%%%%%%%
%%%%%%%%%%%%%%%%%%%%%%%%%%%%%%
\section{Dense ideals and invariance under isomorphisms} \label{Sec:ideals}
%%%%%%%%%%%%%%%%%%%%%%%%%%%%%%
%%%%%%%%%%%%%%%%%%%%%%%%%%%%%%
%%%%%%%%%%%%%%%%%%%%%%%%%%%%%%
%%%%%%%%%%%%%%%%%%%%%%%%%%%%%%

In this section we will see that every vertex in an arbitrary graph connects to a line point, a cycle without exits, an extreme cycle or to a vertex for which its tree has infinitely many bifurcations. 
These different types of vertices: $P_l, P_{c}, P_{ec}$ are related to  ideals which will be the largest in an specific sense, as will be shown in Section 3.

In terms of properties of the associated Leavitt path algebra, the connection to $P_l, P_c, P_{ec}$ and $P_{b^\infty}$ will mean that the ideal generated by $P_l\cup P_c\cup P_{ec}\cup P_{b^\infty}$ is an essential ideal, equivalently, it is a dense ideal of the corresponding Leavitt path algebra.

We prove also that the ideal generated by vertices in an extreme cycle and vertices whose tree has infinitely many bifurcations is invariant under isomorphisms.

We remark the reader that when we speak about isomorphisms, we are considering ring isomorphisms. It was proved in \cite[Proposition 1.2]{KMMS} that if the center of a Leavitt path algebra $L_K(E)$ is isomorphic to $K$, then both concepts coincide. In general, this is not the case.

We start by discussing some properties of the sets that generate the ideals of our concern.

Every Leavitt path algebra has a natural $\Z$-grading given by the length of paths (see \cite[Section 2.1]{AAS}). In a graded algebra over an abelian group, the ideal generated by a set of idempotents of degree zero (where zero is the neutral element in the group) is a graded ideal. In particular, in a Leavitt path algebra $L_K(E)$, the ideals $I( P^E_l)$, $I(P^E_{c})$, $I(P^E_{ec})$ and $I(P_{b^{\infty}}^E)$ are graded.

Recall that $P_l$, $P_c$ and $P_{ec}$ are all hereditary subsets of vertices, however $P_{b^{\infty}}$ is not necessarily hereditary as the following examples show.
\begin{examples}\label{PbinfinityNotHereditary}
\rm
%\begin{enumerate}[(i)]
%\item
(i) Let $E$ be the infinite clock graph having vertices $\{u, v_1, v_2, v_3... \}$ and edges $\{ e_1, e_2, e_3... \}$ with $s(e_i)=v$ and $r(e_i)=v_i$ for all $i= 1,2...$.

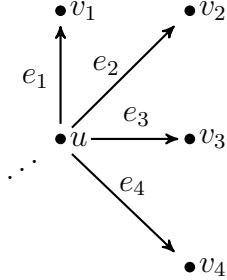
\begin{figure}[H]
	\begin{center}
	\begin{tikzpicture}
	[
	->,
	>=stealth',
	auto,node distance=3cm,
	thick,
	main node/.style={circle, draw, font=\sffamily\Large\bfseries}
	]
	
\fill (0,0) circle (2pt) node[label={[label distance=-1ex]right:{$u$}}] {};
\fill (1.7,0) circle (2pt) node[label={[label distance=-1ex]right:{$v_3$}}]  {};
\fill (0,1.7) circle (2pt) node[label={[label distance=-1ex]right:{$v_1$}}]  {};
\fill (1.7,1.7) circle (2pt) node[label={[label distance=-1ex]right:{$v_2$}}] {};
\fill (1.7,-1.7) circle (2pt) node[label={[label distance=-1ex]right:{$v_4$}}]  {};
%\fill (-1.7,0) circle (2pt) node[label={[label distance=-1ex]left:{$v_n$}}] {};
%\node  at (0,-1.7) {$\cdots$};
\node  at (-0.5,-0.3) {\reflectbox{$\ddots$}};

    \draw[draw=black] (0.4,0) -- (1.55,0) node[above] at (1,0.02){$e_3$} ;
    %\draw[draw=black] (-0.2,0) -- (-1.5,0) node[above] at (-0.8,-0.02) {$e_n$} ;
    \draw[draw=black] (0,0.2) -- (0,1.5) node[left] at (-0.0,0.8) {$e_1$} ;
    \draw[draw=black] (0.15,0.15) -- (1.5,1.5) node[above] at (0.6,0.7) {$e_2$};
    \draw[draw=black] (0.15,-0.20) -- (1.5,-1.5) node[above] at (0.95,-0.9) {$e_4$};
    
	\end{tikzpicture}
	\end{center}
	%\label{fig:sbar}
\end{figure}

Since $u$ is an infinite emitter, it is in $P_{b^{\infty}}$. However $v_i \notin P_{b^{\infty}}$ for any $i$, hence having that
$P_{b^{\infty}}$ is not hereditary.
%
%
%\item
\par
(ii) Let $E$ be the row-finite graph having  vertices $\{v_i, w_i \ \vert \ i = 1, 2, ... \}$, i.e. %and edges  $\{ e_1, e_2, e_3, ... \}$ with $s(e_i)=v$ and $r(e_i)=v_i$ for all $i= 1,2, ...$.

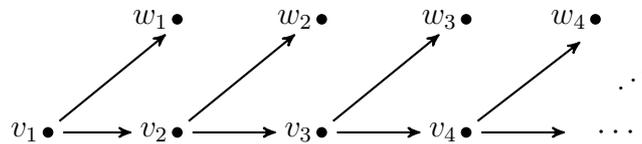
\begin{figure}[H]
	\begin{center}
	\begin{tikzpicture}
	[
	->,
	>=stealth',
	auto,node distance=3cm,
	thick,
	main node/.style={circle, draw, font=\sffamily\Large\bfseries}
	]
	
\fill (0,0) circle (2pt) node[label={[label distance=-1ex]left:{$v_1$}}] {};
\fill (1.7,0) circle (2pt) node[label={[label distance=-1ex]left:{$v_2$}}]  {};
\fill (3.6,0) circle (2pt) node[label={[label distance=-1ex]left:{$v_3$}}]  {};
\fill (5.5,0) circle (2pt) node[label={[label distance=-1ex]left:{$v_4$}}] {};
\fill (1.7,1.5) circle (2pt) node[label={[label distance=-1ex]left:{$w_1$}}]  {};
\fill (3.6,1.5) circle (2pt) node[label={[label distance=-1ex]left:{$w_2$}}] {};
\fill (5.5,1.5) circle (2pt) node[label={[label distance=-1ex]left:{$w_3$}}]  {};
\fill (7.2,1.5) circle (2pt) node[label={[label distance=-1ex]left:{$w_4$}}]  {};
\node  at (7.5,0.0) {$\cdots$};
\node  at (7.7,0.8) {\reflectbox{$\ddots$}};

    \draw[draw=black] (0.2,0) -- (1.1,0) ;
    \draw[draw=black] (0.15,0.15) -- (1.55,1.30) ;
    \draw[draw=black] (1.85,0.15) -- (3.25,1.30) ;
    \draw[draw=black] (1.9,0) -- (3,0) ;
    \draw[draw=black] (3.75,0.15) -- (5.15,1.30) ;
    \draw[draw=black] (3.8,0) -- (4.9,0) ;
    \draw[draw=black] (5.7,0) -- (6.8,0) ;
    \draw[draw=black] (5.65,0.15) -- (7.0,1.20) ;

	\end{tikzpicture}
	\end{center}
	%\label{fig:sbar}
\end{figure}

Then $P_{b^{\infty}}=\{v_i : i = 1, 2, ... \}$ and $P_l= \{w_i: i = 1, 2, ...\}$. Again, $P_{b^{\infty}}$ is not hereditary as $w_i \notin P_{b^{\infty}}$.
%\end{enumerate}
\end{examples}

%\subsection{Dense Ideals}

Dense ideals of a Leavitt path algebra were first studied in \cite{S}.
When the set of vertices of the graph is finite, it is shown that the ideal generated by $P_l\cup P_c\cup P_{ec}$, denoted by $I_{lce}$, is a dense ideal \cite[Theorem 2.9]{CMMSS}. However, this is not the case in general, as the following example shows.
% \cite[Example 3.1.12]{AAS}:

\begin{example}
\rm{Consider the graph $E$:}
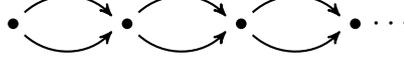
\begin{figure}[H]
	\begin{center}
	\begin{tikzpicture}
	[
	->,
	>=stealth',
	auto,node distance=3cm,
	thick,
	main node/.style={circle, draw, font=\sffamily\Large\bfseries}
	]
	\node  at (-1.2,0) {};
	\fill (0,0) circle (2pt)   {};
    \fill (1.5,0) circle (2pt)   {};
    \fill (3,0) circle (2pt)   {};
    \fill (4.5,0) circle (2pt)   {};
    \node  at (5,0) {$\cdots$};

    \draw [->,draw=black] (0.15,0.15) to [bend right=315, looseness=1] (1.3,0.1);
	\draw [->,draw=black] (0.15,-0.15) to [bend left=315, looseness=1] (1.3,-0.1);
    \draw [->,draw=black] (1.65,0.15) to [bend right=315, looseness=1] (2.8,0.1);
    \draw [->,draw=black] (1.65,-0.15) to [bend left=315, looseness=1] (2.8,-0.1);
	\draw [->,draw=black] (3.15,0.15) to [bend right=315, looseness=1] (4.3,0.1);
	\draw [->,draw=black] (3.15,-0.15) to [bend left=315, looseness=1] (4.3,-0.1);
	
	\end{tikzpicture}
	\end{center}
	%\label{fig:sbar}
\end{figure}
It has neither cycles nor line points, that is,  $P_{ec} = P_c=P_l= \emptyset$. Hence $I_{lce}= {0}$, which is not a dense ideal. Note that $E^0= P_{b^\infty}$.
\end{example}

\medskip

Our aim is to construct a dense ideal for any Leavitt path algebra over an arbitrary graph. To this end we will first find a subset of vertices such that every vertex in the graph connects to it. Then, we will prove that the ideal generated by these vertices is an essential ideal of the Leavitt path algebra. Being essential is equivalent to being dense, as every Leavitt path algebra is left nonsingular and for left nonsingular rings both notions coincide.

%, extending \cite[Proposition 3.7.11]{AAS}. We first generalize \cite[Lemma 3.7.10]{AAS} which states that in a finite graph any vertex connects either to a sink, or to a cycle without exits or to an extreme cycle.

Let $E$ be an arbitrary graph and $H$ a hereditary  subset of $E^0$. The \textit{restriction graph}, denoted by $E_H$, is: 
$$E_H^0:=H, \ E_H^1:=\{e \in E^1 \ \vert \ s(e) \in H\}, $$
where the source and range functions in $E_H$ are simply the source and range functions in $E$ restricted to $H$.

\begin{lemma}\label{Lemma:E^0connectsto} Let $E$ be an arbitrary graph. Then every vertex $v$ connects to at least one of: a line point, a cycle without exits, an extreme cycle, or a vertex whose tree has infinite bifurcations, i.e., every vertex in $E$ connects to $P_l\cup P_c\cup P_{ec}\cup P_{b^{\infty}}$.
\end{lemma}

\begin{proof} Let $X=P_l\cup P_c\cup P_{ec}\cup P_{b^{\infty}}$. For any $v \in E^0$ we will show that $v$ connects to $X$. We distinguish two cases:
\begin{enumerate}
  \item Suppose $|T(v)| < \infty$. Then $H=T(v)$ is a (finite) hereditary subset of $E^0$ and the graph $E_H$ has a finite number of vertices. By  \cite[Lemma 3.7.10]{AAS},  $v$, considered as a vertex in $E_H$, connects to a line point, a cycle without exits, or an extreme cycle. Note that every line point, every cycle without exits and every extreme cycle in $E_H$ is also a line point, a cycle without exits or an extreme cycle, respectively, in $E$; this shows our claim.
  \item Consider $|T(v)| = \infty$. Assume that $T(v) \cap X = \emptyset$, that is, $v$ does not connect to any element in $X$. This means that for any $w \in T(v)$, $w$ is neither a line point, nor a cycle without exits, nor an extreme cycle and it is not in $P_{b^{\infty}}$. First observe that for every $w \in T(v)$ we have $|T(w)| = \infty$ because otherwise $H'=T(w)$ is a finite hereditary subset and applying \cite[Lemma 3.7.10]{AAS} as before to the graph $E_{H'}$, we will have that $w$ connects to a line point, a cycle without exits or an extreme cycle, but this is not possible since we are assuming $T(v) \cap X = \emptyset$.

For $w\in E^0$ define ${\rm Bif}_{T(w)}:=\{u \in E^0 \ | \ u \in T(w) \text{ and there is a bifurcation at } u\}$. We claim that for every $w \in T(v)$ we have $|{\rm Bif}_{T(w)}| \neq 0$. Suppose that for some $w \in T(v)$ we have $|{\rm Bif}_{T(w)}| = 0$. As $w$ is not a line point, $T(w)$ has to contain all the vertices of a cycle $c$, since $T(w) \cap X = \emptyset$ because $T(v) \cap X = \emptyset$. Hence, $c$ has an exit, say $e$, which is a bifurcation in $T(w)$. This is a contradiction. Take $w_1 \in T(v)$. If $w_1 \in P_{b^{\infty}}$ we get a contradiction again with the fact that $T(v) \cap X = \emptyset$. So suppose $w_1 \in T(v)$ and $w_1 \notin P_{b^{\infty}}$; then we know $|T(w_1)| = \infty$ and $0 < |{\rm Bif}_{T(w_1)}| < \infty$. Assume that $T(w_1)$ does not contain any vertex in a cycle; in that case it exists $u_1 \in T(w_1)$ which connects to a line point but this is not possible according to our hypothesis. Therefore, $T(w_1)$ must contain the vertices of a cycle $c_1$, and this cycle has, necessarily, an exit, say $e_1$. Write $r(e_1)=w_2$. Consider $T(w_2)$; then, for the same reasons as before,  $T(w_2)$ has to contain the vertices of a cycle $c_2$, and this cycle must have an exit, say $e_2$. This $r(e_2)$ cannot connect to $c_1$, otherwise we have a vertex that connects to an extreme cycle. If we continue in the same manner, $T(w_1)$ contains infinitely many bifurcations $\{s(e_1), s(e_2), s(e_3) \ldots\}$; but this is a contradiction. This finishes the proof.
\end{enumerate}
\end{proof}

\medskip

A very useful criterion for determining when an ideal is dense is given in \cite[Proposition 1.10]{CMMSS}, which states that for a hereditary subset $H$ of a graph $E$, $I(H)$ is a dense ideal if and only if every vertex of $E$ connects to $H$. Now, Lemma \ref{Lemma:E^0connectsto} gives enough information to determine a dense ideal for every Leavitt path algebra.

\begin{proposition} Let $E$ be an arbitrary graph and
$X= P_l\cup P_c\cup P_{ec}\cup P_{b^{\infty}}$. Then $I(X)$ is a dense ideal.
\end{proposition}

\begin{proof}
By Lemma \ref{Lemma:E^0connectsto}, every vertex connects to $X$ and by \cite[Proposition 1.10]{CMMSS} we are done.
\end{proof}

%%%%%%%%%%%%%%%%%%%%%%%
%%%%%%%%%%%%%%%%%%%%
%%%%%%%%%%%%%%%%%%%%%%%%%
%%%%%%%%%%%%%%%%%%%%%%%%

%\subsection{Invariance of $I(P_{ec} \cup P_{b^{\infty}})$}

In what follows we prove that in an arbitrary graph, the ideal generated by $P_{ec} \cup P_{b^{\infty}}$ is invariant under any ring isomorphism of $L_K(E)$.

%Recall that $P_l$, $P_c$, $P_{ec}$, $P_{b^{\infty}}$ and denote the set of all line points, vertices in cycles with no exits, vertices in extreme cycles and the vertices whose trees contain infinite bifurcations of $E$, respectively.
For any arbitrary graph $E$ the ideal $I(P_{l})$, which is the socle, is invariant under any algebra isomorphism and
$I(P_{c})$ is shown to be invariant under any ring isomorphisms in \cite[Theorem 6.11]{ABS}.
Moreover, it is proven in \cite[Theorem 4.1]{KMMS} that $I(P_{ec})$ remains invariant under any ring isomorphism when $E$ is a finite graph. 
%In general, when there is a ring isomorphism between two algebras, these are not necessarily isomorphic as algebras. However, it is shown that in \cite[Proposition 1.2]{KMMS} that when the center of the Leavitt path algebra is the ground field (ie. when $E$ is a finite graph), then any ring isomorphism gives rise to an algebra isomorphism.

In order to establish Proposition \ref{Proposition:ecpbinfntyinvariant}, we need to see that the ideal $I(P_{ec} \cup P_{b^{\infty}})$ does not contain primitive idempotents. Recall that an idempotent $e$ in an algebra is called \textit{primitive} if $e$ cannot be decomposed as a sum of two non-zero orthogonal idempotents.

\begin{lemma}\label{Lemma:primitiveidempotents}
Let $E$ be an arbitrary graph and $K$ any field. Then
$I(P_{ec} \cup P_{b^{\infty}})$ does not contain any primitive idempotent.
\end{lemma}
\begin{proof}
The graded ideal $I(P_{ec} \cup P_{b^{\infty}})$ is  $K$-algebra isomorphic to a Leavitt path algebra, by  \cite[Corollary 2.5.23]{AAS}; concretely, to the Leavitt path algebra whose underlying graph is $F:= \overline{P_{ec} \cup P_{b^{\infty}}}$ (see the Structure Theorem for graded ideals,  \cite[Theorem 2.5.8]{AAS}). 

The primitive idempotents of the Leavitt path algebra $L_K(F)$ are in the ideal generated by $P^F_l\cup P^F_c$ because the primitive minimal are in the socle of the Leavitt path algebra, which is the ideal generated by $P^F_l$, by \cite[Theorem 5.2]{AMMS2} and the primitive non-minimal are in $I(P^F_c)$, by \cite[Corollary 6.10]{ABS}. Since $L_K(F)$ has neither line points nor cycles without exits, it has no primitive idempotents.
\qedhere
\end{proof}

\begin{proposition}\label{Proposition:ecpbinfntyinvariant} 
Let $E$ be an arbitrary graph. Then the ideal $I(P_{ec} \cup P_{b^{\infty}})$ is invariant under any ring isomorphism.
\end{proposition}

\begin{proof} Assume that $E$ and $F$ are arbitrary graphs and that $\varphi: L_K(E) \rightarrow L_K(F)$ is a ring isomorphism. Note that $I(P_{ec}^E \cup P_{b^{\infty}}^E)$ is generated by idempotents. Since any isomorphism sends idempotents to idempotents, by \cite[Corollary 2.9.11]{AAS}, the ideal $\varphi (I(P_{ec}^E \cup P_{b^{\infty}}^E))$ is  graded. This means that there exists a hereditary saturated set $H$ in $F$ such that $\varphi (I(P_{ec}^E \cup P_{b^{\infty}}^E))= I(H)$ by \cite[Theorem 2.4.8]{AAS}.

Take $v \in H$. By Lemma \ref{Lemma:E^0connectsto}, $v$ connects to a line point, to a cycle without exits, to an extreme cycle or to a vertex whose tree has infinite bifurcations. We are going to show that $v$ can connect neither to a line point nor to a cycle without exits.

If $v$ connects to a line point $w$ then $w \in H$ and $T(w)$ does not have any bifurcations, so $w$ is a primitive idempotent by \cite[Proposition 5.3]{ABS}. Similarly, if $v$ connects to a cycle $c$ without exits, then $c^0 \subseteq H$ and again $H$ contains a primitive idempotent. In both cases, since primitive idempotents are preserved by isomorphisms, $I(P_{ec}^E \cup P_{b^{\infty}}^E)$ contains primitive idempotents but this is a contradiction to Lemma \ref{Lemma:primitiveidempotents}. Hence, $v$ connects either to an extreme cycle or to a vertex whose tree has infinite bifurcations. Assume $v$ connects to a vertex $u$ such that $T(u)$ has infinite bifurcations. Clearly $v \in P_{b^{\infty}}^F$, which means $v \in I(P_{ec}^F \cup P_{b^{\infty}}^F)$.

Suppose that $v$ connects to an extreme cycle. We distinguish the following two cases:

Case 1: There is path $\mu$ starting at $v$ and ending at a vertex of an extreme cycle $c'$, and ${\mu}^0$ contains an infinite emitter $u$. Then $v$ is in $P_{b^{\infty}}^F$.

Case 2: All the paths from $v$ to any extreme cycle contain only regular vertices. Then by (CK2) relation, $v$ is in the ideal $I(P_{ec}^F)$.

Hence $v \in I(P_{ec}^F \cup P_{b^{\infty}}^F)$ and $\varphi (I(P_{ec}^E \cup P_{b^{\infty}}^E)) \subseteq I(P_{ec}^F \cup P_{b^{\infty}}^F)$. Reasoning in the same way with $\varphi^{-1}$ we get  $\varphi^{-1} (I(P_{ec}^F \cup P_{b^{\infty}}^F)) \subseteq I(P_{ec}^E \cup P_{b^{\infty}}^E)$ implying $\varphi (I(P_{ec}^E \cup P_{b^{\infty}}^E)) = I(P_{ec}^F \cup P_{b^{\infty}}^F)$.
\end{proof}

%%%%%%%%%%%%%%%%%%%%%%%%%%%%%%%%%%%%%%%%%%%%%%%%%%%%%%%%%%%%%%%%%%%%%%%%%%%%%%%%%%%%%%%%%%%%%%%%%%%%%%%%%%

\section{The largest locally artinian and locally noetherian ideals of a Leavitt path algebra}\label{Sec:largestlocart&noeth}

To start the picture about  largest ideals generated by the sets of vertices in $P_{lec}^E$, for $E$ an arbitrary graph, we show that there exists a largest semisimple ideal in $L_K(E)$, which is generated by the line points, and a largest locally noetherian ideal, which is generated by vertices in cycles without exits. The notions studied in this section are the following: we say that a ring $R$ is \emph{locally left artinian (resp., locally left
noetherian)} if for any finite subset $X$ of $R$ there exists an idempotent $e \in R$ such that
$X \subseteq eRe$, and $eRe$ is left artinian (resp., left noetherian).

The first statement follows from a general fact that maybe is well-known; we include here because we don't know a concrete reference. 

Recall that for a (non necessarily unital) ring $R$ the \emph{left socle} is defined to be the sum of the minimal left ideals of $R$, while the \emph{right socle} is the sum of the minimal right ideals of $R$. If there are no minimal left (right) ideals, then the left (right) socle is said to be zero. When $R$ is a semiprime ring (i.e., it has no nonzero nilpotent ideals), then the left and the right socle coincide and this ideal is called the \emph{socle} of $R$, denoted ${\rm Soc}(R)$. A left (right) ideal of $R$ will be called \emph{semisimple} if it is semisimple as a left (right) $R$-module, i.e., if $I$ is the sum of simple left (right) $R$-modules.

\begin{proposition}\label{prop:BigSemis}
Let $R$ be a semiprime ring. Then the socle is the largest semisimple left (and right) ideal of $R$.
\end{proposition}
\begin{proof}
Denote by $S$ the socle of $R$ and let $I$ be a semisimple left ideal. Then $I$ is a direct sum of simple left ideals of $R$, say $I=\oplus_{i\in \Lambda}I_i$. Since $R$ is semiprime, $I_i=Re_i$, being $e_i$ an idempotent in $I$ which is minimal, i.e., $e_iRe_i$ is a division ring. Apply that the socle is the sum of all minimal ideals to get that $I$ must be contained in $S$, as required. 
\end{proof}

\begin{theorem}\label{thm:BigSoc}
Let $E$ be an arbitrary graph and let $K$ be a field. Then  $I(P_l)$ is the largest semisimple left and right ideal of $L_K(E)$. It is also the largest locally left and right artinian ideal of the Leavitt path algebra.
\end{theorem}
\begin{proof}
Apply Proposition \ref{prop:BigSemis}, \cite[Proposition 2.3.1]{AAS} and \cite[Theorem 2.6.14]{AAS}.
\end{proof}

Our next goal is to show that the ideal generated by the set of line points jointly with the vertices which lie in cycles without exits is the largest left/right locally noetherian ideal of a Leavitt path algebra. As a result we will obtain that the ideal generated by $P_c$ is the largest locally left/right noetherian ideal not having minimal idempotents.

One of the key points in the proof will be the Structure Theorem for graded ideals in a Leavitt path algebra, which is proved in \cite[Theorem 2.5.8]{AAS}. Here we include some of the notions involved in this result. 

Let $E$ be an arbitrary graph and $K$ any field. Given a hereditary subset $H$ of $E^0$ and a vertex $v\in E^0$, we say that $v$ is a \emph{breaking vertex of} $H$ if $v$ is in the set
$$B_H:= \{v \in E^0\setminus H \ \vert \ v \in {\rm Inf}(E)\ \text{and} \
\vert s^{-1}(v) \cap r^{-1}(E^0\setminus H)\vert < \infty\}.$$

In other words, $B_H$ consists of those vertices of $E$ which are infinite emitters, which are not in $H$, and for which the ranges of the edges they emit are all, except for a finite (and nonzero) number, inside $H$ (see \cite[Definitions 2.4.4]{AAS}). For $v \in B_H$, recall that the element $v^H$ of $L_K(E)$ is:
$$v^H:= v- \sum_{e \in s^{-1}(v) \cap r^{-1}(E^0\setminus H)}ee^*$$
For any subset $S \subseteq B_H$, define $S^H:= \{ v^H \ | \  v \in S\} \subseteq L_K(E)$.

Also we need to recall here the definition of the generalized hedgehog graph of a hereditary set (\cite[Definition 2.5.20]{AAS}). Let $H$ be a hereditary subset of $E^0$ and $S \subseteq B_H$. Define the \textit{generalized hedgehog graph of} $H$ as follows:
$$
F_1(H,S):=\{\alpha=e_1\cdots e_n \in \rm{Path}(E) \ \vert \ r(e_n)\in H ;\ s(e_n)\notin H \cup S\}, and $$
$$F_2(H,S)=\{\alpha=e_1\cdots e_n \in \rm{Path}(E) \ \vert \ n \geq 1; \ r(e_n)\in S\} .$$
For $i=1,2$, denote by $\overline{F}_{i}(H,S)$ another copy of $F_{i}(H,S)$; for any $\alpha\in F_{i}(H,S)$ we will write $\overline{\alpha}$ to refer copy of $\alpha$ in $\overline{F}_{i}(H,S)$. Define a new graph $_{(H,S)}E=((_{(H,S)}E)^0,(_{(H,S)}E)^1,s',r')$ as follows: $$(_{(H,S)}E)^0=H\cup S \cup F_{1}(H,S) \cup F_{2}(H,S) \text{ and} $$
$$(_{(H,S)}E)^1=\{e\in E^1 \mid s(e)\in H\} \cup \{e\in E^1 \mid s(e)\in S; \ r(e) \in H\}\cup \overline{F}_{1}(H,S) \cup \overline{F}_{2}(H,S).$$ 

The source and range maps $s'$ and $r'$ are defined by extending $r$ and $s$ to $_{(H,S)}E^1$ and by setting $s'(\overline{\alpha})=\alpha$ and $r'(\overline{\alpha})=r(\alpha)$ for all $\overline{\alpha}\in \overline{F}_{i}(H,S)$ for $i=1,2$. In the particular case in which $S= \emptyset$, we have that $F_2(H,\emptyset) = \emptyset$ and $_{(H,\emptyset)}E = _HE$ given in \cite[Definition 2.5.16]{AAS}.

\begin{theorem}\label{Them:LargestLLN} Let $E$ be an arbitrary graph and let $K$ be any field. Then 
$I(P_l\sqcup P_c)$ is the largest locally left (right) noetherian ideal of $L_K(E)$.
\end{theorem}
\begin{proof}
By \cite[Corolary 2.7.5 (i)]{AAS}, \cite[Lemma 4.2.2 (ii)]{AAS} and \cite[Lemma 4.2.4]{AAS} we have that $I(P_l\sqcup P_c$) is locally left noetherian.
Now we prove that it is the largest locally left (right) noetherian ideal.
Let $I$ be an ideal of $L_K(E)$ which is locally left noetherian. By the definition of left locally noetherian, $I$ is generated as an ideal by the idempotents it contains, so it is a graded ideal. By the Structure Theorem of graded ideals \cite[Theorem 2.5.8]{AAS}, $I=I(H\cup S^H)$ for $H=I\cap E^0$. 

Next we claim that $I$ does not contain elements of the form $v^H$, for $v$  a breaking vertex. Assume on the contrary $v^H\in I$ and take  an infinite subset $\{f_i\;|\;i\in\mathbb{N}\} \subseteq s^{-1}(v)\cap r^{-1}(H)$. Then we have the following increasing chain inside $v^{H}Iv^{H}$: \[v^{H}L_{K}(E)f_1f_1^*v^{H}\subsetneq v^{H}(L_K(E)f_1f_1^*\oplus  L_K(E)f_2f_2^*)v^{H}\subsetneq\cdots\]
This is a contradiction because $v^{H}Iv^{H}$ is a left noetherian algebra (every corner of a locally left noetherian algebra is left noetherian).

Then we know that $I=I(H)$ and, by \cite[Theorem 2.5.19]{AAS}, we have $I(H)\cong L_K(_HE)$ which is locally left noetherian. We know that $L_K(_HE)=I(P_l^{_HE}\sqcup P_c^{_HE})$ by \cite[Theorem 4.2.12]{AAS}. We claim that $P_l^{_HE}\sqcup P_c^{_HE}$ can be seen inside $I(P_l^E\sqcup P_c^E)$. Indeed, if $p\in P_l^{_HE}$, then $p$ is a line point in $H$ or $p$ comes from a path in $E$ ending at a vertex in $H$ which is a line point in $E^0$; by abuse of notation we denote this path by $p$. Then  $p \in I(P_l^{E})$. On the other hand, every cycle without exits in $_HE$ comes from a cycle without exits in $E$; this means that we may assume $P_c^{_HE}\subseteq P_c^E$ (understanding the containment as a graph homomorphism as defined in \cite[Definition 1.6.1]{AAS}). This shows $I\subseteq I(P_l^{E}\sqcup P_c^{E})$ as required. \qedhere
\end{proof}

\begin{corollary}
For an arbitrary graph $E$ and any field $K$, the ideal $I(P_c)$ is the largest locally left/right noetherian ideal not having minimal idempotents
\end{corollary}
\begin{proof}
Apply Theorem \ref{Them:LargestLLN} and that every minimal idempotent is in the socle of $L_K(E)$, which is generated by the vertices in $P_l^E$ (see \cite[Theorem 2.6.14]{AAS}).
\end{proof}

%%%%%%%%%%%%%%%%%%%%%%%%%%%%%%
%%%%%%%%%%%%%%%%%%%%%%%%%%%%%%
\section{The largest purely infinite ideal of a Leavitt path algebra}
%%%%%%% SECTION
%%%%%%%%%%%%%%%%%%%%%%%%%%%%%%
%%%%%%%%%%%%%%%%%%%%%%%%%%%%%%
%%%%%%%%%%%%%%%%%%%%%%%%%%%%%%
%%%%%%%%%%%%%%%%%%%%%%%%%%%%%%
%%%%%%%%%%%%%%%%%%%%%%%%%%%%%%
%%%%%%%%%%%%%%%%%%%%%%%%%%%%%%
%\textcolor{magenta}{the definition of a hedgehog graph.}
In this section we show that any purely infinite ideal in a Leavitt path algebra is graded and we find the largest purely infinite ideal of the algebra, which happens to be the ideal generated by the properly purely infinite vertices. 

We start by recalling the definition of purely infinite ring, that (without simplicity) was introduced in \cite[Definition 3.1]{AGPS}. A ring $R$ is said to be \textit{purely infinite} if the following conditions are satisfied:

\begin{enumerate}
    \item No quotient of $R$ is a division ring, and
    \item whenever $a \in R$ and $b \in RaR$, then $b=xay$ for some $x, y \in R$.
\end{enumerate}

A vertex $v$ in an arbitrary graph is called {\it properly infinite} if and only if there exist vertices $w_1, w_2,..., w_n$ in $T(v)$ such that $|{\rm CSP}(w_i)| \geq 2$ for all $i$ and
$v \in \overline{\{ w_1, w_2,..., w_n \}}$ (see \cite[Proposition 3.8.12]{AAS}). The set of properly infinite vertices of a Leavitt path algebra will be denoted by $P_{pi}$, or by $P_{pi}^E$ if we want to emphasize the graph we are considering.

%Let $E$ be an arbitrary graph and $K$ any field. Given a hereditary subset $H$ of $E^0$ and a vertex $v\in E^0$, we say that $v$ is a \emph{breaking vertex of} $H$ if $v$ is in the set
%$$B_H:= \{v \in E^0\setminus H \ \vert \ v \in {\rm Inf}(E)\ \text{and} \
%\vert s^{-1}(v) \cap r^{-1}(E^0\setminus H)\vert < \infty\}.$$

%In words, $B_H$ consists of those vertices of $E$ which are infinite emitters, which are not in $H$, and for which the ranges of the edges they emit are all, except for a finite (and nonzero) number, inside $H$ (see \cite[Definitions 2.4.4]{AAS}). 

Leavitt path algebras which are purely infinite can be characterized as those whose graph satisfies a nice property, as stated in \cite[Corollary 3.8.17]{AAS}: every vertex is properly infinite and there are no breaking vertices for any hereditary subset of vertices of the graph. This is the result that follows.

\begin{proposition}\label{prop:pilpa}
Let $E$ be an arbitrary graph and $K$ be any field. The following are equivalent:
\begin{enumerate} [\rm(i)]
	\item $L_{K}(E)$ is purely infinite.
	\item $B_H=\emptyset$ for all $H\in \mathcal{H}_E$, and every vertex is properly infinite.
\end{enumerate}
\end{proposition}

In order to determine the largest purely infinite ideal of a Leavitt path algebra, we first study which type of ideal it must be.
\medskip

\begin{proposition}\label{Proposition:piisgraded} Let $E$ be an arbitrary graph and $K$ any field. Then every purely infinite ideal $I$ of $L_K(E)$ is graded. Moreover, there exists a hereditary and saturated subset $H\subseteq E^0$ such that $I=I(H)$.
\end{proposition}

\begin{proof} Let $I$ be a nonzero purely infinite ideal of $L_K(E)$. By \cite[Theorem 2.8.10]{AAS} we have that $I=I(H \cup S^H \cup P_C)$ where $H, S^H$ and $P_C$ are as described therein. If $I$ were not graded, then $P_C \neq \emptyset$ and the ideal $I/I(H \cup S^H)$ of $L_K(E)/I(H \cup S^H)$ would be isomorphic to $\bigoplus_{\overline{c}\in \overline{C}}M_{\Lambda_{\overline{c}}}(p_c(x)K[x,x^{-1}])$. Observe that this algebra is  not purely infinite. To see this it is enough to show that $p_c(x)K[x,x^{-1}]$ is not purely infinite. Indeed,  if $<x>$ is the ideal generated by $x$ in $K[x,x^{-1}]$, there exists a quotient of $p_c(x)K[x,x^{-1}]$, concretely $p_c(x)K[x,x^{-1}] / p_c(x)<x>$, which is isomorphic to the field $K$, so (1) in \cite[Definitions 3.8.3 (ii)]{AAS} is not satisfied. This fact contradicts the purely infiniteness of $I/I(H \cup S^H)$ (by \cite[Lemma 3.8.9 (i)]{AAS}) and, consequently, $I=I(H\cup S^{H})$, i.e., $I$ is graded. 

Apply \cite[Theorem 2.5.22]{AAS} to get that $I$ is ($K$-algebra) isomorphic to the Leavitt path algebra $L_{K}(_{(H,S)}E)$. Now we prove that $S=\emptyset$. Assume on the contrary that there is an element $u\in S$. Since $u$ is a breaking vertex of $H$ in $E$, it is an infinite emitter and emits infinitely many edges into $H$ in the graph $E$. By the construction of the generalized hedgehog graph, the vertex  $u$ is an infinite emitter in $_{(H,S)}E$ and $|{\rm {\rm CSP}}(u)|=0$, also in $_{(H,S)}E$. This implies that $u$ is not a properly infinite vertex in $_{(H,S)}E$, contradicting that $L_{K}(_{(H,S)}E)$ is purely infinite. Therefore  $S=\emptyset$ and $I=I(H)$ as desired. 
\end{proof}

It is shown in \cite[Corollary 2.9.11]{AAS} that an ideal in a Leavitt path algebra is itself a Leavitt path algebra if and only if it is a graded ideal. The corresponding Leavitt path algebra is the one associated to the generalized hedgehog graph of a certain hereditary set (\cite[Theorem 2.5.22]{AAS}). 

%Here we 
%recall the definition of the hedgehog graph of a hereditary set. 

%Let $H$ be a hereditary subset of $E^0$. Define the \textit{hedgehog graph of} $H$ as:
%$$
%F_{E}(H)=\{\alpha=e_1\cdots e_n\ \in \rm{Path}(E) \ \vert \ s(\alpha)\in E^0\backslash H,\;r(e_i)\in E^0\backslash H\;for\;all\;1\leq i<n,\;r(\alpha)\in H\}.$$
%Denote by $\overline{F}_{E}(H)$ another copy of $F_{E}(H)$, for any $\alpha\in F_{E}(H)$ we will write $\overline{\alpha}$ to refer copy of $\alpha$ in $\overline{F}_{E}(H)$. Define a new graph $_HE=((_HE)^0,(_HE)^1,s',r')$ as follows: $$(_HE)^0=H\cup F_{E}(H)\;\;and\;\;(_HE)^1=\{e\in E^1 \mid s(e)\in H\}\cup \overline{F}_{E}(H)$$ The source and range functions $s'$ and $r'$ are defined by setting $s'(e)=s(e)$ and $r(e')=r(e)$ for every $e\in E^1$ such that $s(e)\in H$ and by setting $s'(\overline{\alpha})=\alpha$ and $r'(\overline{\alpha})=r(\alpha)$ for all $\overline{\alpha}\in \overline{F}_{E}(H)$. 

Since the ideal generated by an extreme cycle is purely infinite (see\cite[Lemma 2.5]{CMMSS}), a question that naturally arises is whether a purely infinite Leavitt path algebra has to contain extreme cycles. The answer is no, as the following example shows.

\begin{example}
\rm The Leavitt path algebra of the following graph is purely infinite but has no extreme cycles.
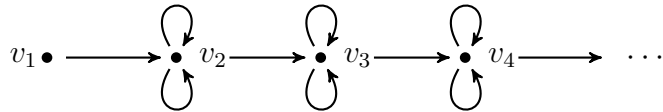
\begin{figure}[H]
	\begin{center}
	\begin{tikzpicture}
	[
	->,
	>=stealth',
	auto,node distance=3cm,
	thick,
	main node/.style={circle, draw, font=\sffamily\Large\bfseries}
	]
	
	\fill (0,0) circle (2pt) node[label={[label distance=-1ex]left:{$v_1$}}] (1) {};
    \fill (1.7,0) circle (2pt) node[label=right:{$v_2$}] (2) {};
    \fill (3.6,0) circle (2pt) node[label=right:{$v_3$}] (3) {};
    \fill (5.5,0) circle (2pt) node[label=right:{$v_4$}] (4) {};
    \draw[draw=black] (0.25,0) -- (1.5,0) node[above] at (0.9,0.03) {};
    \path (2) edge [loop above, looseness=12,in=60, out=120] (2);
    \path (2) edge [loop below, looseness=12,in=-60, out=-120](2);
    \path (3) edge [loop above, looseness=12, in=60, out=120] (3);
    \path (3) edge [loop below,  looseness=12, in=-60, out=-120] (3);
    \path (4) edge [loop above,  looseness=12, in=60, out=120] (4);
    \path (4) edge [loop below,  looseness=12, in=-60, out=-120] (4);
    \draw [->,draw=black] (2.4,0)--(3.4,0) ;

	\draw [->,draw=black] (4.3,0)--(5.3,0);
	
	\draw [->,draw=black] (6.2,0)--(7.3,0);
	
	\node  at (7.9,0) {$\cdots$};
	%\node  at (8.9,0) {$\cdots$};
	\end{tikzpicture}
	\end{center}
	%\label{fig:sbar}
\end{figure}
\end{example}

On the other hand, the ideal generated by the set of all vertices in extreme cycles is a purely infinite ideal.

\begin{theorem}\label{TheoremPecpurelyinf}
Let $E$ be an arbitrary graph and $K$ a field. Then $I(P_{ec})$ is a purely infinite ideal.
\end{theorem}
\begin{proof}
Recall that $P_{ec}$ is a hereditary set and denote it by $H$. By \cite[Theorem 2.5.19]{AAS} the ideal generated by $H$ is $K$-algebra isomorphic to the Leavitt path algebra of the hedgehog graph ${}_HE$. We will use (ii) in Proposition \ref{prop:pilpa}. We prove that the two conditions in (ii) are satisfied.

(i) Assume on the contrary that there exists a hereditary saturated set $Y \subseteq {}_HE^0$ with $B_Y \neq \emptyset$. Take $v \in B_Y$. Since $v$ is an infinite emitter, by the construction of the hedgehog graph $v \notin F_E(H)$, so
$v \in H=P_{ec}$. Moreover, as $v \in B_Y$, $v \notin Y$. There exists an edge $e$ starting from $v$ to a vertex $u$ in $Y$. As $H$ is hereditary, $u \in H$. Also, $e$ is either in the extreme cycle where $v$ lies on, or $e$ is an exit for the extreme cycle to which $v$ belongs. In both cases, there is a path from $u$ to $v$. Hence, $v \in Y$. This is a contradiction.

(ii) Let $v \in {}_HE^0$. If $v \in H$, we can take $w_1=v$ and since $v$ is a vertex in an extreme cycle then
$|{\rm CSP}(v)| \geq 2$ is satisfied.
Suppose $v \in F_E(H)$, then $v$ corresponds to a path $\alpha=e_1e_2...e_n$ in $E$, where
$s(e_1) \in E^0\backslash H$, $r(e_i) \in E^0\backslash H$ for all $1 \leq i < n$ and $r(e_n) \in H$.
There is an edge $\overline{v}$ in the hedgehog graph ${}_HE$ such that $r(\overline{v})=r(e_n):= w \in H$. Since
$w$ is a vertex in an extreme cycle, $|{\rm CSP}(w)| \geq 2$ is satisfied.
Moreover, in the hedgehog graph ${}_HE$, $w \in T(v)$ and $v \in \overline{\{w\}}$.
\end{proof}

Next, we want to investigate whether $I(P_{ec})$ is the largest purely infinite ideal in $L_K(E)$.
Note that in a ring $R$ with local units, if $R$ is purely infinite then any ideal $I$ of $R$ is also purely infinite.
Moreover, $R/I$ is also a purely infinite ring \cite[Lemma 3.8.9]{AAS}.
Hence, if $L_K(E)$ is a purely infinite ring, then any ideal is purely infinite. The examples that follow illustrate that $I(P_{ec})$ is not necessarily the largest purely infinite ideal. 

\medskip

\begin{example}
\rm
Consider the graph $E$:

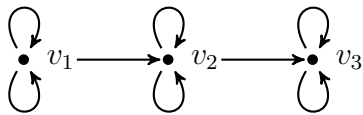
\begin{figure}[H]
	\begin{center}
	\begin{tikzpicture}
	[
	->,
	>=stealth',
	auto,node distance=3cm,
	thick,
	main node/.style={circle, draw, font=\sffamily\Large\bfseries}
	]
	\fill (1.7,0) circle (2pt) node[label=right:{$v_1$}] (2) {};
    \fill (3.6,0) circle (2pt) node[label=right:{$v_2$}] (3) {};
    \fill (5.5,0) circle (2pt) node[label=right:{$v_3$}] (4) {};
    \path (2) edge [loop above, looseness=12,in=60, out=120] (2);
    \path (2) edge [loop below, looseness=12,in=-60, out=-120](2);
    \path (3) edge [loop above, looseness=12, in=60, out=120] (3);
    \path (3) edge [loop below,  looseness=12, in=-60, out=-120] (3);
    \path (4) edge [loop above,  looseness=12, in=60, out=120] (4);
    \path (4) edge [loop below,  looseness=12, in=-60, out=-120] (4);
    \draw [->,draw=black] (2.4,0)--(3.5,0) ;

	\draw [->,draw=black] (4.3,0)--(5.4,0);
		
	\end{tikzpicture}
	\end{center}
	%\label{fig:sbar}
\end{figure}

\vspace{0.5cm}

The Leavitt path algebra $L_K(E)$ is purely infinite. Both $P_{ec} =\{ v_3\}$ and $\{v_2,v_3\}$ are hereditary sets that generate proper purely infinite ideals with $I(\{v_2,v_3\}) \supsetneq I(P_{ec})$.
\end{example}

\begin{example}\label{kezban}
\rm
Consider the graph $E$:

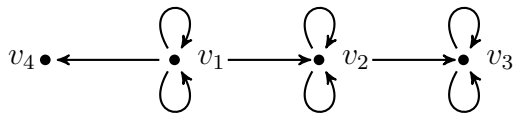
\begin{figure}[H]
	\begin{center}
	\begin{tikzpicture}
	[
	->,
	>=stealth',
	auto,node distance=3cm,
	thick,
	main node/.style={circle, draw, font=\sffamily\Large\bfseries}
	]
	
	\fill (0,0) circle (2pt) node[label={[label distance=-1ex]left:{$v_4$}}] (1) {};
    \fill (1.7,0) circle (2pt) node[label=right:{$v_1$}] (2) {};
    \fill (3.6,0) circle (2pt) node[label=right:{$v_2$}] (3) {};
    \fill (5.5,0) circle (2pt) node[label=right:{$v_3$}] (4) {};
    \draw[draw=black] (1.5,0) -- (0.15,0);
    \path (2) edge [loop above, looseness=12,in=60, out=120] (2);
    \path (2) edge [loop below, looseness=12,in=-60, out=-120](2);
    \path (3) edge [loop above, looseness=12, in=60, out=120] (3);
    \path (3) edge [loop below,  looseness=12, in=-60, out=-120] (3);
    \path (4) edge [loop above,  looseness=12, in=60, out=120] (4);
    \path (4) edge [loop below,  looseness=12, in=-60, out=-120] (4);
    \draw [->,draw=black] (2.4,0)--(3.5,0) ;

	\draw [->,draw=black] (4.3,0)--(5.4,0);
		
	\end{tikzpicture}
	\end{center}
	%\label{fig:sbar}
\end{figure}
\vspace{0.25cm}
The Leavitt path algebra $L_K(E)$ is not purely infinite. The ideal generated by the vertices in extreme cycles, $I(P_{ec})$, is purely infinite, but it is not the largest one as it is strictly contained in the purely infinite ideal $I(\{v_2,v_3\})$. 
\end{example}

% We continue our search for purely infinite ideals in a Leavitt path algebra among the ideals generated by various subsets of vertices of $E$. First, we consider the set of all properly infinite idempotents. 
% Let $P_{pi}^E$ be the set of all properly infinite idempotents in $L_{K}(E)$ (for simplicity, we will omit the graph $E$ in the notation and use $P_{pi}$.) 

\begin{lemma}\label{lem:ECinPI} For an arbitratry graph $E$ and any field $K$, we have that
 $P_{ec}^E \subseteq P_{pi}^E$. 
 \end{lemma}
\begin{proof}
Let $u$ be a vertex in an extreme cycle, and take $v\in T(u)$. By the definition of extreme cycle there exists $w \in T(v)$ with $|{\rm CSP}(w)| \geq 2$. This implies, by \cite[Lemma 3.8.11]{AAS}, that $u$ is a properly infinite idempotent.
\end{proof}

The set of properly infinite vertices $P_{pi}$, is not necessarily a hereditary set.

\begin{example}
\rm
In the graph  
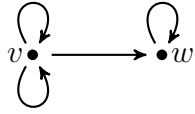
\begin{figure}[H]
	\begin{center}
	\begin{tikzpicture}
	[
	->,
	>=stealth',
	auto,node distance=3cm,
	thick,
	main node/.style={circle, draw, font=\sffamily\Large\bfseries}
	]
	
	\fill (0,0) circle (2pt) node[label={[label distance=-1ex]left:{$v$}}](1) {};
    \fill (1.7,0) circle (2pt) node[label={[label distance=-1ex]right:{$w$}}] (2) {};
    \draw[draw=black] (0.25,0) -- (1.5,0);
    \path (1) edge [loop above, looseness=12,in=60, out=120] (1);
    \path (1) edge [loop below, looseness=12,in=-60, out=-120](1);
    \path (2) edge [loop above, looseness=12, in=60, out=120] (2);
    		
	\end{tikzpicture}
	\end{center}
	%\label{fig:sbar}
\end{figure}
%\vspace{0.5cm}
\noindent
the vertex $v$ is in $P_{pi}$, but $w$ is not a properly infinite vertex and $v \geq w$.
\end{example}

\begin{example}
\rm
Consider the graph $E$ in Example \ref{kezban} and denote by $e$ the edge starting at $v_1$ and finishing at $v_4$. 
%We prove that the ideal generated by the set of properly infinite vertices is not purely infinite.
%
%\begin{figure}[H]
%	\begin{center}
%	\begin{tikzpicture}
%	[
%	->,
%	>=stealth',
%	auto,node distance=3cm,
%	thick,
%	main node/.style={circle, draw, font=\sffamily\Large\bfseries}
%	]
%	
%	\fill (0,0) circle (2pt) node[label={[label %distance=-1ex]left:{$v_4$}}] (1) {};
%    \fill (1.7,0) circle (2pt) node[label=right:{$v_1$}] (2) {};
%    \fill (3.6,0) circle (2pt) node[label=right:{$v_2$}] (3) {};
 %   \fill (5.5,0) circle (2pt) node[label=right:{$v_3$}] (4) {};
 %   \draw[draw=black] (1.5,0) -- (0.25,0) node[above] at (0.9,0.03) %{$e$};
 %   \path (2) edge [loop above, looseness=12,in=60, out=120] (2);
 %   \path (2) edge [loop below, looseness=12,in=-60, out=-120](2);
 %   \path (3) edge [loop above, looseness=12, in=60, out=120] (3);
%    \path (3) edge [loop below,  looseness=12, in=-60, out=-120] %(3);
%    \path (4) edge [loop above,  looseness=12, in=60, out=120] (4);
%    \path (4) edge [loop below,  looseness=12, in=-60, out=-120] %(4);
%    \draw [->,draw=black] (2.4,0)--(3.5,0) ;
%
%	\draw [->,draw=black] (4.3,0)--(5.4,0);
%%		
%	\end{tikzpicture}
%	\end{center}
%	%\label{fig:sbar}
%\end{figure}
%
%\vspace{0.5cm}

We know that $L_K(E)$ is not a purely infinite ring. Observe that $P_{pi} = \{v_1,v_2,v_3\}$ and $v_4 \in I(\{v_1,v_2,v_3\})$ since $v_4=e^*v_1e$. So $I(\{v_1,v_2,v_3\})=L_K(E)$, which is not purely infinite.
\end{example}

Our next aim is to provide a subset of vertices which  will generate the largest purely infinite ideal of a Leavitt path algebra.
Define: 

\[P_{ppi}:= \{v\in E^{0}\mid T(v)\subseteq P_{pi}\;\text{and} \;T(v)\;\text{has\;no\;breaking\;vertices}.\}\]

\begin{lemma}\label{lem:tpiher}
Let $E$ be an arbitrary graph. Then:
\begin{enumerate} [\rm(i)]
\item $P_{ppi}$ is a hereditary and saturated set.
\item $P_{ec}\subseteq P_{ppi}$.
\end{enumerate}

\begin{proof}
(i) Let $v\in P_{ppi}$ and $w\in T(v)$. Since $T(w)\subseteq T(v)$, $T(w)\subseteq P_{pi}$; apply that there are no breaking vertices in $T(v)$ and therefore in $T(w)$, to get $w\in P_{ppi}$. This shows that it is hereditary. That $P_{ppi}$ is saturated follows immediately.

(ii) Let $v\in P_{ec}$, where $v\in c^0$ for some extreme cycle $c$. Take $w\in T(v)$.  Let  $\alpha$ be a path such that $s(\alpha)=v$ and $r(\alpha)=w$. Since $v$ is in an extreme cycle, there exists another path $\beta$ starting at $w$ and ending at a vertex in $c^0$. By the definition of extreme cycle, $|{\rm CSP}(w)|\geq 2$ and so $w\in P_{pi}$; using that there are no breaking vertices in $T(v)$ we obtain $v\in P_{ppi}$. 
\end{proof}
\end{lemma}

\begin{proposition}\label{Prop:PPIispurelyinfinite}
Let $E$ be an arbitrary directed graph and $P_{ppi}$ be the set defined above. The ideal $I(P_{ppi})$ is purely infinite.
\end{proposition}

\begin{proof}
	Denote $H:=P_{ppi}$, which is a hereditary and saturated set by Lemma \ref{lem:tpiher}. Apply \cite[Theorem 2.5.19]{AAS} to get that $I(H)\cong L_{K}(_HE)$. We will show that the Leavitt path algebra $L_{K}(_HE)$ is purely infinite using Proposition \ref{prop:pilpa}. Note that the hedgehog graph $_HE$ has no breaking vertices since the same happens to $H$. Therefore (i) in Proposition \ref{prop:pilpa} is satisfied.
	
	Now, take $v\in (_HE)^0$; if $v\in H$, that is, $T(v)\subseteq P_{pi}$, then $v\in P_{pi}$ and we are done. If $v\in F_{E}(H)$ then there is only one edge starting at $v$ and ending at a vertex $w\in H$. Since $w\in P_{pi}$, there exist $w_1, w_2, ..., w_n$ in $T(w)$ such that $|{\rm CSP}(w_i)| \geq 2$ for all $i$ and
$w \in \overline{\{ w_1, w_2, ..., w_n \}}$. Clearly $w_1, w_2, ..., w_n$ in $T(v)$ and $v \in \overline{\{ w_1, w_2, ..., w_n \}}$. This proves (ii) in Proposition \ref{prop:pilpa}.
\end{proof}

\begin{theorem}\label{Theorem:piisthemaximal}
	Let $E$ be an arbitrary directed graph. The ideal $I(P_{ppi})$ is the largest purely infinite ideal in $L_{K}(E)$.
\end{theorem}
\begin{proof}

Let $J=I(H)$ be a purely infinite ideal of $L_{K}(E)$, where $H$ is a hereditary and saturated subset of $E^0$ by Proposition \ref{Proposition:piisgraded}. Apply \cite[Theorem 2.5.19]{AAS} to get that $I(H)\cong L_{K}(_HE)$. Our aim is to show $H\subseteq P_{ppi}$. 

Take $v\in H$. Then $v$ is properly infinite and its tree in $_{H}E$ has no breaking vertices. Hence there exist $w_1, w_2, ..., w_n \in T^{_{H}E}(v)$  such that $|{\rm CSP}^{_{H}E}(w_i)| \geq 2$ for all $i$ and
$v \in {\overline{\{ w_1, w_2, ..., w_n \}}^{_{H}E}}$. By the construction of the hedgehog graph, $w_1, w_2, ..., w_n \in T^{E}(v)$ and $|{\rm CSP}^{E}(w_i)| \geq 2$ for all $i$; besides,
$v \in \overline{\{ w_1, w_2, ..., w_n \}}^{E}$. Therefore, $v$ is a properly infinite vertex in $E$. Moreover, in the graph $E$, its tree has no breaking vertices. Since $L_{K}(_HE)$ is purely infinite, we have $T^{E}(v)\subseteq P_{pi}$. So, $v\in P_{ppi}$ and we conclude that the ideal $I(P_{ppi})$ is the largest purely infinite ideal in $L_{K}(E)$.
\end{proof}

The condition that $T(v)$ does not contain breaking vertices cannot be eliminated in order to have a purely infinite ideal. Define
\[P'_{ppi}:= \{v\in E^{0}\mid T(v)\subseteq P_{pi}\}\]
The example that follows shows that the ideal $I(P'_{ppi})$ is not  purely infinite.

\begin{example}
\rm 
Consider the graph $E$:

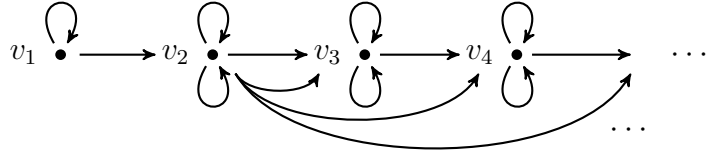
\begin{figure}[H]
	\begin{center}
	\begin{tikzpicture}
	[
	->,
	>=stealth',
	auto,node distance=3cm,
	thick,
	main node/.style={circle, draw, font=\sffamily\Large\bfseries}
	]
	
	\fill (0,0) circle (2pt) node[label=left:{$v_1$}] (1) {};
    \fill (2,0) circle (2pt) node[label=left:{$v_2$}] (2) {};
    \fill (4,0) circle (2pt) node[label=left:{$v_3$}] (3) {};
    \fill (6,0) circle (2pt) node[label=left:{$v_4$}] (4) {};
    \node  at (8.3,0) {$\cdots$};
    
    \node  at (7.5,-1) {$\cdots$};
    \draw[->,draw=black] (0.25,0) -- (1.25,0);
    \path (1) edge [loop above, looseness=12, in=60, out=120] node[above] {} (1);
    \path (2) edge [loop above, looseness=12, in=60, out=120] (2);
    \path (2) edge [loop below, looseness=12, in=-60, out=-120] (2);
    
    \path (3) edge [loop above, looseness=12, in=60, out=120] (3);
    \path (3) edge [loop below, looseness=12, in=-60, out=-120] (3);
    
    \path (4) edge [loop above, looseness=12, in=60, out=120] (4);
    \path (4) edge [loop below, looseness=12, in=-60, out=-120] (4);    
    
    \draw [->,draw=black] (2.2,0) -- (3.25,0);
    
	\draw [->,draw=black] (4.20,0) -- (5.25,0);
	\draw [->,draw=black] (6.20,0) -- (7.5,0);
	
	\draw [->,draw=black] (2.3,-0.25) to [bend left=300, looseness=0.8] (3.4,-0.25);
	\draw [->,draw=black] (2.3,-0.25) to [bend left=300, looseness=0.75] (5.5,-0.25);
	\draw [->,draw=black] (2.3,-0.25) to [bend left=300, looseness=0.75] (7.5,-0.25);
	
	\end{tikzpicture}
	\end{center}
	%\label{fig:sbar}
\end{figure}

$P'_{ppi}=\{v_2,v_3,\cdots\}$ and the corresponding hedgehog graph $_{P'_{ppi}}E$ is:

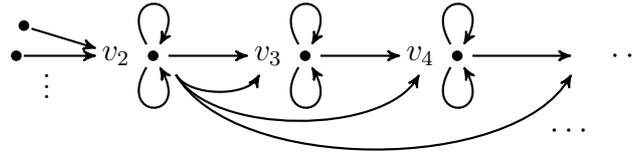
\begin{figure}[H]
	\begin{center}
	\begin{tikzpicture}
	[
	->,
	>=stealth',
	auto,node distance=3cm,
	thick,
	main node/.style={circle, draw, font=\sffamily\Large\bfseries}
	]

    \fill (2,0) circle (2pt) node[label=left:{$v_2$}] (2) {};
    \fill (4,0) circle (2pt) node[label=left:{$v_3$}] (3) {};
    \fill (6,0) circle (2pt) node[label=left:{$v_4$}] (4) {};
    
    \fill (0.20,0.0) circle (2pt) node[label=left:{}] (1) {};
    \fill (0.30,0.40) circle (2pt) node[label=left:{}] (1) {};
    \node  at (8.3,0) {$\cdots$};
    
    \node  at (7.5,-1) {$\cdots$};
    \node  at (0.6,-0.3) {$\vdots$};
    %\node  at (1,0.5) {$\cdots$};
        
    \draw[->,draw=black] (0.40,0.36) -- (1.25,0.1);
    \draw[->,draw=black] (0.3,0.0) -- (1.25,0.0);
    %\path (1) edge [loop above, looseness=12, in=60, out=120] node[above] {$e$} (1);
    \path (2) edge [loop above, looseness=12, in=60, out=120] (2);
    \path (2) edge [loop below, looseness=12, in=-60, out=-120] (2);
    
    \path (3) edge [loop above, looseness=12, in=60, out=120] (3);
    \path (3) edge [loop below, looseness=12, in=-60, out=-120] (3);
    
    \path (4) edge [loop above, looseness=12, in=60, out=120] (4);
    \path (4) edge [loop below, looseness=12, in=-60, out=-120] (4);
    
    \draw [->,draw=black] (2.2,0) -- (3.25,0);
    
	\draw [->,draw=black] (4.20,0) -- (5.25,0);
	\draw [->,draw=black] (6.20,0) -- (7.5,0);
	
	\draw [->,draw=black] (2.3,-0.25) to [bend left=300, looseness=0.8] (3.4,-0.25);
	\draw [->,draw=black] (2.3,-0.25) to [bend left=300, looseness=0.75] (5.5,-0.25);
	\draw [->,draw=black] (2.3,-0.25) to [bend left=300, looseness=0.75] (7.5,-0.25);
	
	\end{tikzpicture}
	\end{center}
	%\label{fig:sbar}
\end{figure}

The set $Y=\{v_3,v_4,\cdots\}$ is hereditary and saturated in the graph $_{P'_{ppi}}E$ and clearly $B_Y=\{v_2\}$, therefore $I(P'_{ppi})\cong L_{K}(_{P'_{ppi}}E)$ is not purely infinite.
\end{example}

\begin{corollary}\label{Proposition:ppiinvariant}
Let $E$ be an arbitrary graph. Then the ideal $I(P_{ppi})$ is invariant under any ring isomorphism.
\end{corollary}

\begin{proof} Suppose that $E$ and $F$ are arbitrary graphs and that $\varphi: L_K(E) \rightarrow L_K(F)$ is a ring isomorphism. Denote $I:=I(P_{ppi}^E)$ and $I':=I(P_{ppi}^F)$.

First we show that $\varphi(I) \subseteq I'$. To have this, it is enough to prove that $\varphi(I)$ is a purely infinite ideal in $L_K(F)$ because of Theorem \ref{Theorem:piisthemaximal}. We check that the following two conditions (in the definition of purely infinite ring) are satisfied:

\begin{enumerate}
    \item No quotient of $\varphi(I)$ is a division ring, and
    \item whenever $a' \in \varphi(I)$ and $b' \in \varphi(I)a'\varphi(I)$, then $b'=x'a'y'$ for some $x', y' \in \varphi(I)$.
\end{enumerate}
For the first one, suppose on the contrary that there exits a quotient of $\varphi(I)$ which is a division ring, say $\varphi(I) / \varphi(J)$ where $J$ is an ideal of $I$. Since $\overline{\varphi} : I/J \rightarrow \varphi(I)/ \varphi(J)$ is an isomorphism then $I/J$ is a quotient of $I$ which is a division ring so we get a contradiction to the fact that $I$ is purely infinite.

For the second condition take $a' \in \varphi(I)$ and $b' \in \varphi(I)a'\varphi(I)$, and let $a \in I$ and $b \in L_K(E)$ be such that $\varphi(a)=a'$ and $\varphi(b)=b'$. Then $\varphi(b)\in \varphi(I)\varphi(a)\varphi(I)=\varphi(IaI)$, which implies $b\in IaI$. Now,  being $I$ purely infinite means that $b=xay$ for some $x, y\in I$. Then, taking $x'=\varphi(x)$ and $y'=\varphi(y)$ we obtain $b'=x'a'y'$.

Analogously we get $\varphi^{-1}(I') \subseteq I$ and, therefore,  $\varphi(I) = I'$ as desired.
\end{proof}

\section{The structure of the largest purely infinite ideal}\label{Sec:structurelpi}
In the previous section we established the existence of the largest purely infinite ideal of a Leavitt path algebra. 
The aim of this section is to deep into its structure. Concretely, we will prove that it is the direct sum of purely infinite simple ideals and purely infinite non-simple indecomposable ideals.
We start with some definitions we need.

\begin{definitions}
 \rm
 From the set of vertices in extreme cycles and from the set of vertices which are properly infinite, we pick up the following:
$$P_{ec}':=\{v \in P_{ec} \ \vert \ \text{there exists}\  u\in P_{ppi}\setminus P_{ec} \ \text{such that}\ u\geq v\}.$$
A cycle whose vertices are in $P_{pec}$ will be called a \emph{properly extreme cycle}. Note that extreme cycles are divided into two sets: those whose vertices are properly  infinite and the complement.
$$P_{pec}= P_{ec}\setminus P_{ec}'.$$
In the set of properly  infinite vertices, we remove those belonging to properly extreme cycles and denote it by $P'$, i.e.,
$$P':= P_{ppi} \setminus P_{pec}.$$
\end{definitions}

Cycles whose vertices are in $P'$ will produce (graded) ideals which are purely infinite and non-simple (moreover, we will see that they are also non-decomposable). The question which arises is how to relate cycles of this type which are in the same purely infinite ideal.   This is the reason because we establish the  relations given in the definitions below.

\begin{definitions}
 \rm
% \begin{enumerate}[\rm (i)]
(i) (See \cite[Definitions 2.2]{CMMSS}).
Let $X'_{pec}$ be the set of all cycles whose vertices are in $P_{pec}$. We define in $X'_{pec}$ the following relation: given $c, d \in X'_{pec}$, we write $c \sim' d$ if $c$ and $d$ are connected. This is an equivalence relation. Denote the set of all equivalence classes by $X_{pec}= X'_{pec}/\sim$. If we want to emphasize the graph we are considering we write $X'_{pec}(E)$ and $X_{pec}(E)$ for $X'_{pec}$ and $X_{pec}$, respectively.
 
For any $c\in X'_{pec}$, let $\widetilde{c}$ denote the class of $c$, and use $\widetilde{c}^0$ to represent the set of all vertices belonging to the cycles which are in $\widetilde{c}$.

(ii)  Let $X'_{P'}$ be the set of all cycles whose vertices are in $P'$. We define in $X'_{P'}$ the following relation: given $c, d \in X'_{P'}$, we write $c \sim' d$ if $c$ and $d$ are connected. This relation is reflexive and symmetric, but not necessarily transitive. Now we define in $X'_{P'}$ the relation: $c \sim d$ if there are $c_1,\dots,  c_n\in X'_{P'}$ such that $c=c_1 \sim' c_2 \sim' \dots \sim' c_n=d$. This is an equivalence relation. Denote the set of all equivalence classes by $X_{P'}= X'_{P'}/\sim$. If we want to emphasize the graph we are considering we write $X'_{P'}(E)$ and $X_{P'}(E)$ for $X'_{P'}$ and $X_{P'}$, respectively.
 
For any $c\in X'_{P'}$, let $\widetilde{c}$ denote the class of $c$, and use $\widetilde{c}^0$ to represent the set of all vertices belonging to the cycles which are in $\widetilde{c}$.
 \end{definitions}

\begin{example}
\rm
Consider the following graph:
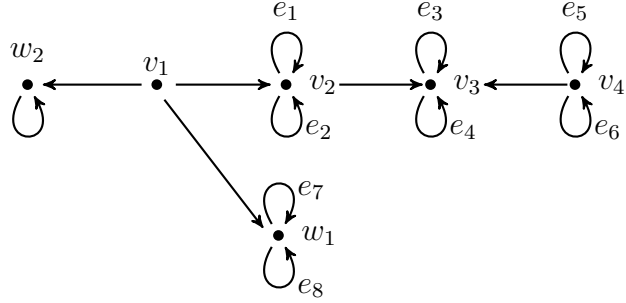
\begin{figure}[H]
	\begin{center}
	\begin{tikzpicture}
	[
	->,
	>=stealth',
	auto,node distance=3cm,
	thick,
	main node/.style={circle, draw, font=\sffamily\Large\bfseries}
	]
	
	\fill (0,0) circle (2pt) node[label={[label distance=-1ex]above:{$v_1$}}] (1) {};
    \fill (1.7,0) circle (2pt) node[label=right:{$v_2$}] (2) {};
    \fill (3.6,0) circle (2pt) node[label=right:{$v_3$}] (3) {};
    \fill (5.5,0) circle (2pt) node[label=right:{$v_4$}] (4) {};
    \fill (1.6,-2.0) circle (2pt) node[label=right:{$w_1$}] (5) {};
    \fill (-1.7,0) circle (2pt) node[label=above:{$w_2$}] (6) {};
    \draw[draw=black] (0.25,0) -- (1.5,0) node[above] at (0.9,0.03) {};
    
    \path (2) edge [loop above, looseness=12,in=60, out=120] (2) node[above] at (1.7,0.7) {$e_1$};
    
    \path (2) edge [loop below, looseness=12,in=-60, out=-120](2) node[right] at (1.8,-0.6) {$e_2$};
    \path (3) edge [loop above, looseness=12, in=60, out=120] (3) node[above] at (3.6,0.7) {$e_3$};;
    \path (3) edge [loop below,  looseness=12, in=-60, out=-120] (3) node[right] at (3.7,-0.6) {$e_4$};
    \path (4) edge [loop above,  looseness=12, in=60, out=120] (4)node[above] at (5.5,0.7) {$e_5$};
    \path (4) edge [loop below,  looseness=12, in=-60, out=-120] (4) node[right] at (5.6,-0.6) {$e_6$};
    \path (5) edge [loop above,  looseness=12, in=60, out=120] (5)node[right] at (1.7,-1.4) {$e_7$};;
    \path (5) edge [loop below,  looseness=12, in=-60, out=-120] (5) node[right] at (1.7,-2.7) {$e_8$};
    \path (6) edge [loop below,  looseness=12, in=-60, out=-120] (6);
    
    \draw [->,draw=black] (2.4,0)--(3.5,0) ;

	\draw [->,draw=black] (5.4,0)--(4.3,0);
	
	\draw [->,draw=black] (0.1,-0.2)--(1.4,-1.9);
	\draw [->,draw=black] (-0.2,0)--(-1.5,0);

	\end{tikzpicture}
	\end{center}
	%\label{fig:sbar}
\end{figure}
Then, $P_{ec}=\{v_3,w_1\}$, $P_{ppi}=\{v_2,v_3,v_4,w_1\}$, $P'_{ec}=\{v_3\}$, $P_{pec}=\{w_1\}$ and  $P'=\{v_2,v_3,v_4\}$.
Moreover,
$X'_{P'}=\{e_1,e_2,e_3,e_4,e_5,e_6\}$ and $e_1\sim e_2\sim e_3\sim e_4\sim e_5\sim e_6 $, so $X_{P'}=\{\widetilde{e_1}\}$. Finally, note that
$\widetilde{e_1}^0=\{v_2,v_3,v_4\}$.

\end{example}

\begin{remark}
\rm
Let $E$ be an arbitrary graph and use the definitions before.
\begin{enumerate}[\rm (i)]
    \item For any $c\in X'_{P'}$ we have that $\widetilde{c}^0$ is not necessarily a hereditary subset (see Example \ref{ExRem}).
    \item Given $c, d\in X'_{P'}$ we have that $\widetilde{c}\neq \widetilde{d}$ if and only if $\widetilde{c}^0\cap \widetilde{d}^0=\emptyset$.
\end{enumerate}
\end{remark}

\begin{example}\label{ExRem}
\rm
Consider the graph:
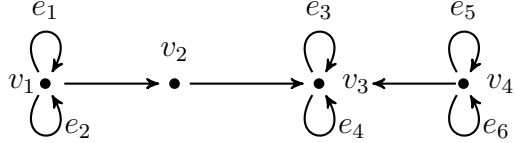
\begin{figure}[H]
	\begin{center}
	\begin{tikzpicture}
	[
	->,
	>=stealth',
	auto,node distance=3cm,
	thick,
	main node/.style={circle, draw, font=\sffamily\Large\bfseries}
	]
	
	\fill (0,0) circle (2pt) node[label={[label distance=-1ex]left:{$v_1$}}] (1) {};
    \fill (1.7,0) circle (2pt) node[label=above:{$v_2$}] (2) {};
    \fill (3.6,0) circle (2pt) node[label=right:{$v_3$}] (3) {};
    \fill (5.5,0) circle (2pt) node[label=right:{$v_4$}] (4) {};
    
    \draw[draw=black] (0.25,0) -- (1.5,0) node[above] at (0.9,0.03) {};
    
    \path (1) edge [loop above, looseness=12,in=60, out=120] (1) node[above] at (0,0.7) {$e_1$};
    \path (1) edge [loop below, looseness=12,in=-60, out=-120](1) node[right] at (0.1,-0.6) {$e_2$};
    
    \path (3) edge [loop above, looseness=12, in=60, out=120] (3) node[above] at (3.6,0.7) {$e_3$};;
    \path (3) edge [loop below,  looseness=12, in=-60, out=-120] (3) node[right] at (3.7,-0.6) {$e_4$};
    
    \path (4) edge [loop above,  looseness=12, in=60, out=120] (4)node[above] at (5.5,0.7) {$e_5$};
    \path (4) edge [loop below,  looseness=12, in=-60, out=-120] (4) node[right] at (5.6,-0.6) {$e_6$};

    \draw [->,draw=black] (1.9,0)--(3.4,0) ;

	\draw [->,draw=black] (5.4,0)--(4.3,0);
	
	\end{tikzpicture}
	\end{center}
	%\label{fig:sbar}
\end{figure}
Then, $P_{ec}=\{v_3\}$, $P_{ppi}=\{v_1,v_2,v_3,v_4\}$, $P'_{ec}=\{v_3\}$, $P_{pec}=\emptyset$ and  $P'=\{v_1,v_2,v_3,v_4\}$.
Moreover,
$X'_{P'}=\{e_1,e_2,e_3,e_4,e_5,e_6\}$ and $e_1\sim e_2\sim e_3\sim e_4\sim e_5\sim e_6 $, so $X_{P'}=\{\widetilde{e_1}\}$. Note that
$\widetilde{e_1}^0=\{v_1,v_3,v_4\}$, but $v_2\notin \widetilde{e_1}^0$.
\end{example}

The result that follows describes each piece into which the ideal generated by $P'$ decomposes.

\begin{proposition}\label{prop:NonSimplePI}
Let $E$ be an arbitrary graph and $K$ a field. For every cycle $c\in X'_{P'}$, the ideal $I(\widetilde{c}^0)$ is isomorphic to a purely infinite non simple Leavitt path algebra which is not decomposable. Concretely, it is isomorphic to $L_K({_H}E)$, where $H=T(\widetilde{c}^0)$. 
\end{proposition}
\begin{proof}
Since every vertex in $H$ is properly infinite and $B_H=\emptyset$, by \cite[Theorem 3.8.16]{AAS}, the Leavitt path algebra $L_K({_H}E)$ is purely infinite.

To see that $L_K({_H}E)$ is not simple, equivalently, that $I(\widetilde{c}^0)$ is not simple, choose a non-extreme cycle $d$ such that $\widetilde{d}= \widetilde{c}$. Take $e\in E^1$ such that $u:=s(e)\in d^0$ but $T(u) \cap c^0 = \emptyset$. Then $0\neq I(r(e))\subsetneq I(\widetilde{c}^0)$.

Our next aim is to prove that $L_K(_{H}E)$ is not decomposable showing that condition (a) in \cite[Theorem 5.2 (iii)]{CMMS} is not satisfied. Let $Y'$ be a nonempty, proper, hereditary and saturated subset of $(_{H}E)^0$. We claim that there is a nonempty hereditary and saturated subset $Y$ of $E^0$ such that $Y\subseteq \overline{H}$. Indeed, the ideal generated by $Y'$ in $L_K(_{H}E)$, call it $J'$, is nonzero, graded and proper. Using the isomorphism between  $L_K(_{H}E)$ and $I(H)$ (see \cite[Theorem 2.5.22]{AAS}), we can say that $J'$ is isomorphic to a graded ideal of $I(H)$, call it $J$, which does not contain breaking vertices. Since $I(H)$ is a ring with local units, then $J$ is a graded ideal of $L_K(E)$ without breaking vertices. By the Structure Theorem for  Graded Ideals \cite[Theorem 2.5.8]{AAS} there exists a nonempty, hereditary and saturated set $Y\in E^0$ such that $J=I(Y)$; moreover, since $J\subseteq I(H)$, we have $\emptyset \neq Y\subsetneq H$ (see \cite[Proposition 2.5.4]{AAS}). We are going to prove:

$(\ast)$ There is a cycle $d$ in $L_K(E)$ such that $d^0\not\subseteq Y$ and $d$ is connected to $Y$.

Let $u\in Y\subsetneq \overline{H}$. Recall that $\overline{H} = \cup_{n\geq 0} H_n$. If $u\in H_0$ then $u\in \widetilde{c}^0$. Let $d$ be a cycle such that $\widetilde{d}=\widetilde{c}$ and there exists an edge $e$ satisfying $s(e)\in d^0$ and $r(e)=u$. Then $ddd...$ is an infinite path whose vertices are not contained in $Y$ and $d$ is connected to $Y$. If $u\in H_1$, then $u$ is the source of an edge $f$ such that $r(f)\in \widetilde{c}^0$. As before, we can find a cycle $d$, with
 $\widetilde{d}=\widetilde{c}$, and an edge $g$ satisfying $s(g)\in d^0$ and $r(g)=v$. Then $ddd...$ is an infinite path whose vertices are not contained in $Y$ and $d$ is connected to $v$, which is in $Y$ because $Y$ is hereditary and $u\in Y$. By induction we prove the statement.
 
 Once we have that $(\ast)$ is true, this provides an infinite path $ddd...$ in $_HE$ whose vertices are outside from $Y'$ but all of them are connected to $Y'$. This means that (2) in  \cite[Theorem 5.2 (iii)]{CMMS} is not satisfied and, therefore, $L_K(_HE)$ is not decomposable.

\end{proof}
\begin{example}
\rm 
In Example \ref{ExRem}, the ideal $I(\widetilde{e_1}^0)$ is isomorphic to $L_K({_H}E)$, where $H=\{v_2,v_3,v_4\}=P'$ and the graph ${_H}E$ is:

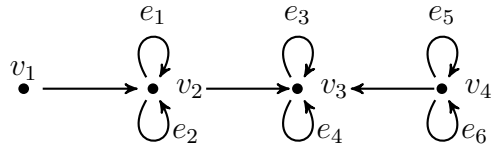
\begin{figure}[H]
	\begin{center}
	\begin{tikzpicture}
	[
	->,
	>=stealth',
	auto,node distance=3cm,
	thick,
	main node/.style={circle, draw, font=\sffamily\Large\bfseries}
	]
	
	\fill (0,0) circle (2pt) node[label={[label distance=-1ex]above:{$v_1$}}] (1) {};
    \fill (1.7,0) circle (2pt) node[label=right:{$v_2$}] (2) {};
    \fill (3.6,0) circle (2pt) node[label=right:{$v_3$}] (3) {};
    \fill (5.5,0) circle (2pt) node[label=right:{$v_4$}] (4) {};
    
    \draw[draw=black] (0.25,0) -- (1.5,0) node[above] at (0.9,0.03) {};
    
    \path (2) edge [loop above, looseness=12,in=60, out=120] (2) node[above] at (1.7,0.7) {$e_1$};
    
    \path (2) edge [loop below, looseness=12,in=-60, out=-120](2) node[right] at (1.8,-0.6) {$e_2$};
    \path (3) edge [loop above, looseness=12, in=60, out=120] (3) node[above] at (3.6,0.7) {$e_3$};;
    \path (3) edge [loop below,  looseness=12, in=-60, out=-120] (3) node[right] at (3.7,-0.6) {$e_4$};
    \path (4) edge [loop above,  looseness=12, in=60, out=120] (4)node[above] at (5.5,0.7) {$e_5$};
    \path (4) edge [loop below,  looseness=12, in=-60, out=-120] (4) node[right] at (5.6,-0.6) {$e_6$};

    \draw [->,draw=black] (2.4,0)--(3.5,0) ;

	\draw [->,draw=black] (5.4,0)--(4.3,0);

	\end{tikzpicture}
	\end{center}
	%\label{fig:sbar}
\end{figure}
$L_K({_H}E)$ is purely infinite and non simple Leavitt path algebra which is not decomposable.
\end{example}

In the following result we are using the notation introduced in \cite[Definitions 2.2]{CMMSS}.

\begin{theorem}\label{theor:biggestPIideal}
Let $E$ be an arbitrary graph and $K$ be a field. Then
$I(P_{ppi}) =  I(P_{pec}) \oplus I(P')$. Moreover, 
$$I(P_{pec}) = \oplus_{\widetilde{c}\in X_{pec}}I(\widetilde{c}^0) \text{ and } I(P') = \oplus_{\widetilde{c}\in X_{P'}}I(\widetilde{c}^0),$$
where every $I(\widetilde{c}^0)$ for $\widetilde{c}\in X_{pec}$  is isomorphic to a Leavitt path algebra which is purely infinite simple, and every $I(\widetilde{c}^0)$ for $\widetilde{c}\in X_{P'}$  is isomorphic to a Leavitt path algebra which is purely infinite non simple and non decomposable.
\end{theorem}
\begin{proof}
Decompose $P_{ppi}=P_{pec}\sqcup {P'}$. Then \cite[Proposition 2.4.7]{AAS} implies $I(P_{ppi})=I(P_{pec})\oplus I(P')$.
By \cite[Proposition 2.6]{CMMSS} we have that $I(P_{pec})= \oplus_{\widetilde{c}\in X_{pec}}I({\widetilde c}^0)$, where every $I({\widetilde c}^0)$ is purely infinite and simple. By Proposition \ref{prop:NonSimplePI} we have $I(P') = \oplus_{\widetilde{c}\in X_{P'}}I(\widetilde{c}^0)$, where each $I(\widetilde{c}^0)$ is purely infinite non simple and non decomposable. This completes the proof.
\end{proof}

\begin{corollary}\label{cor:pii} Let $I$ be a purely infinite ideal of a Leavitt path algebra $L_K(E)$.
\begin{enumerate}[\rm(i)]
    \item If $I$ simple, then it is contained in $I(P_{pec})$. More concretely, $I=I(\widetilde{c}^0)$, for $c$ an extreme cycle such that $c^0\subseteq P_{pec}$.
    \item If $I$ is not simple and not decomposable, then it is contained in $I(P')$. More concretely, $I=I(\widetilde{c}^0)$, where  $c$ is an extreme cycle such that  $c^0\subseteq P'$.
    \item $I=\oplus_{i\in \Lambda}I(\widetilde{c_i}^0)$, where $c_i$ is an extreme cycle such that $c_i^0\subseteq P_{pec}$ or $c_i^0\subseteq P'$.
\end{enumerate}
\end{corollary}

\begin{proof} 
Since $I$ is a purely infinite ideal of $L_K(E)$ and $I(P_{ppi})$ is the largest purely infinite ideal in the Leavitt path algebra (Theorem \ref{Theorem:piisthemaximal}), then $I \subseteq I(P_{ppi})$. By Theorem \ref{theor:biggestPIideal} we may write $I(P_{ppi})= I(P_{pec}) \oplus I(P')$. Moreover, $I=I(H)$ for some hereditary and saturated subset $H \subseteq E^0$ by Proposition \ref{Proposition:piisgraded}. Using \cite[Theorem 2.5.8]{AAS} we have $H \subseteq P_{pec} \sqcup P'$.

We know that $P_{pec} = \sqcup_{\widetilde{c} \in X_{pec}} \widetilde{c}^0$, where every $I(\widetilde{c}^0)$ is purely infinite simple, and that $P' = \sqcup_{\widetilde{c} \in X_{P'}} \widetilde{c}^0$, where every $I(\widetilde{c}^0)$ is purely infinite non simple and non decomposable. Apply this to get $H= \sqcup_{i\in \Lambda} \widetilde{c_i}^0$,  where $c_i$ is an extreme cycle such that $c_i^0\subseteq P_{pec}$ or $c_i^0\subseteq P'$. Now, use \cite[Proposition 2.4.7]{AAS} to get  $I=\oplus_{i\in \Lambda}I(\widetilde{c_i}^0)$. This proves (iii).

If $I$ is simple or indecomposable, then $\vert \Lambda \vert=1$. This implies $I$ is as stated in (i) when $I$ is simple, or as stated in (ii) if it is indecomposable and non simple.
\end{proof}

\begin{corollary}
Let $E$ be an arbitrary graph. Then the ideal $I(P^E_{ec})$ is invariant under any ring isomorphism between Leavitt path algebras.
\end{corollary}
\begin{proof}
Assume that $\varphi: L_K(E) \to L_K(F)$ is a ring isomorphism. By \cite[Proposition 2.6]{CMMSS} we have that $I(P^E_{ec})= \oplus_{\widetilde{c}\in X_{ec}}I(\widetilde{c}^0)$, where every $I(\widetilde{c}^0)$ is purely infinite and simple and $X_{ec}$ is as defined in \cite[Definition 2.2]{CMMSS}. Take a cycle $c$ such that $\widetilde{c}\in X_{ec}$. Since $\varphi$ is an isomorphism and $I(\widetilde{c}^0)$ is purely infinite simple, then $\varphi(I(\widetilde{c}^0))$ is a purely infinite simple ideal of $L_K(F)$. Moreover, by Proposition \ref{Proposition:ppiinvariant}, we get $\varphi(I(P_{ppi}^E)) = I(P_{ppi}^F)$. Since $P_{ec}^E\subseteq P_{ppi}^F$, then $\varphi(I(\widetilde{c}^0))\subseteq I(P_{ppi}^F)$. Use Theorem \ref{theor:biggestPIideal} to get $\varphi(I(\widetilde{c}^0))= I(\widetilde{d}^0)$, where $d$ is a cycle in $X_{pec}^F$. Then $d$ must be an extreme cycle by (i) in Corollary \ref{cor:pii}. Therefore $I(\widetilde{d}^0)\subseteq I(P_{ec}^F)$ and, consequently, our claim follows.
\end{proof}

%%%%%%%%%%%%%%%%%%%%%%%%%%%%%%%%%%%%%%%%%%%%%%%%%%

\section{The largest exchange ideal of a Leavitt path algebra}\label{Sec:largestexchange}

In this section we will describe graphically the largest exchange ideal of a Leavitt path algebra, which exists by \cite[Theorem 3.5]{AMM}.

%First we need to define the following sets. Consider
%
%\[P_{1}:=\{v\in c^0\mid c\;\text{is a cycle with exits and}\; |{\rm CSP}(s(c))|=1\}.\]

%Let $U$ be the hereditary closure of $P_c \cup P_1$, and let
%$$P_{ex}:= E^0 \setminus \left(F^0_E(U)\cup U \right).$$
%It is not difficult to prove that $P_{ex}$ is a hereditary subset of $E^0$. 
\begin{definition}
\rm
 Let $E$ be an arbitrary graph and $H$ a hereditary subset of $E^0$. We say that $H$ \emph{satisfies Condition} (K) if the restriction graph $E_H$ satisfies Condition (K).
\end{definition}

For an arbitrary graph $E$ we consider the  set
$$P^E_{(K)}:= \{v\in E^0 \ \vert \ T(v) \ \hbox{satisfies Condition (K)}\}.$$
It is clear that $P_{(K)}$ is a hereditary subset of vertices. We define $P^E_{ex}$ as

$$P^E_{ex}:= P_{(K)}\cup B^{P_{(K)}}.$$
When it is clear the graph we are considering, we simply write $P_{(K)}$ and $P_{ex}$.

\begin{theorem}\label{thm:largestex}
Let $E$ be an arbitrary graph and $K$ be a field. Then the largest exchange ideal of the Leavitt path algebra $L_K(E)$ is $I(P_{ex})$.
\end{theorem}
\begin{proof} The ideal $I(P_{ex})$ is graded as it is generated by elements of degree zero in the Leavitt path algebra $L_K(E)$, therefore $I(P_{ex})\cong L_K(_{(P_{(K)}, B_{P_{(K)}})}E)$ by \cite[Theorem 2.5.22]{AAS}. By the definition of $P_{ex}$, it is clear that the hedgehog graph $_{(P_{(K)}, B_{P_{(K)}})}E$ satisfies Condition $(K)$, therefore, $L_K(_{(P_{(K)}, B_{P_{(K)}})}E)$ is an exchange ring, by \cite[Proposition 3.3.11]{AAS} and, consequently, the ideal $I(P_{ex})$ is exchange.

Now, let $I$ be the largest exchange ideal of $L_K(E)$. We prove that it is a graded ideal. By the structure theorem of ideals \cite[Theorem 2.8.10]{AAS}, $I=I(H \cup S^H \cup P_C)$, for $H$, $S^H$ and $P_C$ as described in the mentioned theorem. If $P_C\neq \emptyset$, then $I/I(H\cup S^H)$, which is an exchange ring by \cite[Theorem 2.2]{P}, is a $K$-algebra isomorphic to a direct sum of matrices over an ideal of $K[x, x^{-1}]$. But this is not an exchange ring. Consequently, $P_C=\emptyset$ and $I$ is graded. By \cite[Theorem 2.5.22]{AAS} we obtain that $I$ is isomorphic to the Leavitt path algebra $L_{K}(_{(H,S)}E)$. Since it is exchange, the graph $_{(H,S)}E$ satisfies Condition (K) by \cite[Theorem 3.3.11]{AAS}.   This implies $H \cup S^H \subseteq P_{ex}$ and, therefore,  $I= I(H \cup S^H) \subseteq I(P_{ex})$.
\end{proof}

\begin{corollary}
Let $E$ be an arbitrary graph. Then the ideal $I(P_{ex})$ is invariant under any ring isomorphism.
\end{corollary}

\begin{proof}
Assume that $E$ and $F$ are arbitrary graphs and let $\varphi: L_K(E) \rightarrow L_K(F)$ be a ring isomorphism. Denote $I:= I(P_{ex}^{E})$ and $J:=I(P_{ex}^{F})$. Using the definition of exchange ring without unit given in \cite[Theorem 1.2 (c)]{P}, it is clear that  $\varphi(I)$ is an exchange ideal in $L_K(F)$. 
 Since $J$ is the largest exchange ideal in $L_K(F)$, by Theorem \ref{thm:largestex}, we get $\varphi(I)\subseteq J$. Applying the same to $\varphi^{-1}$ we have $\varphi^{-1}(J)\subseteq I$, thus we obtain  $J=\varphi(I)$.
\end{proof}

\end{document}